\documentclass{amsart}
\usepackage{amsfonts,amscd,amsthm,amsgen,amsmath,amssymb}
\usepackage[all]{xy}
\usepackage{epsfig,color}
\usepackage[vcentermath]{youngtab}
\newtheorem{theorem}{Theorem}[section]
\newtheorem{lemma}[theorem]{Lemma}

\newtheorem{proposition}[theorem]{Proposition}

\theoremstyle{definition}
\newtheorem{definition}[theorem]{Definition}
\newtheorem{example}[theorem]{Example}

\theoremstyle{remark}
\newtheorem{remark}[theorem]{Remark}

\numberwithin{equation}{section}

\begin{document}

\setlength\parskip{0.5em plus 0.1em minus 0.2em}

\title[An adjunction inequality for Real embedded surfaces]{An adjunction inequality for Real embedded surfaces}
\author{David Baraglia}

\address{School of Mathematical Sciences, The University of Adelaide, Adelaide SA 5005, Australia}

\email{david.baraglia@adelaide.edu.au}

%\subjclass[2010]{Primary 53C08, 19L50; Secondary 53D18, 53C80}

\date{\today}

%%%%%%%%%%%%%%%%%%%%%%%%%%%%%%%%%%%%%%%%%%%%%%%%%%%%%%%%%%%%%%%%%%%%%%%%%%%%%%%%
%%%%%%%%%%%%%%%%%%%%%%%%%%%%%%%%%%%%%%%%%%%%%%%%%%%%%%%%%%%%%%%%%%%%%%%%%%%%%%%%
%%%%%%%%%%%%%%%%%%%%%%%%%%%%%%%%%%%%%%%%%%%%%%%%%%%%%%%%%%%%%%%%%%%%%%%%%%%%%%%
%%%%%%%%%%%%%%%%%%%%%%%%%%%%%%%%%%%%%%%%%%%%%%%%%%%%%%%%%%%%%%%%%%%%%%%%%%%%%%%%
\begin{abstract}
A Real structure on a $4$-manifold $X$ is an orientation preserving smooth involution $\sigma$. We say that an embedded surface $\Sigma \subset X$ is Real if $\sigma$ maps $\Sigma$ to itself orientation reversingly. We prove that a cohomology class $u \in H^2(X ; \mathbb{Z})$ can be represented by a Real embedded surface if and only if $u$ can be lifted to a class in equivariant cohomology $H^2_{\mathbb{Z}_2}(X ; \mathbb{Z}_-)$. We prove that if the Real Seiberg--Witten invariants of $X$ are non-zero then the genus of Real embedded surfaces in $X$ satisfy an adjunction inequality. We prove two versions of the adjunction inequality, one for non-negative self-intersection and one for arbitrary self-intersection. We show with examples that the minimal genus of Real embedded surfaces can be larger than the minimal genus of arbitrary embedded surfaces.
\end{abstract}

\maketitle

%%%%%%%%%%%%%%%%%%%%%%%%%%%%%%%%%%%%%%%%%%%%%%%%%%%%%%%%%%%%%%%%%%%%%%%%%%%%%%%%
%%%%%%%%%%%%%%%%%%%%%%%%%%%%%%%%%%%%%%%%%%%%%%%%%%%%%%%%%%%%%%%%%%%%%%%%%%%%%%%%
%%%%%%%%%%%%%%%%%%%%%%%%%%%%%%%%%%%%%%%%%%%%%%%%%%%%%%%%%%%%%%%%%%%%%%%%%%%%%%%%
%%%%%%%%%%%%%%%%%%%%%%%%%%%%%%%%%%%%%%%%%%%%%%%%%%%%%%%%%%%%%%%%%%%%%%%%%%%%%%%%

%%%%%%%%%%%%%%%%%%%%%%%%%%%%%%%%%%%%%%%%%%

%%%%%%%%%%%%%%%%%%%%%%%%%%%%%%%

%%%%%%%%%%%%%%%%%%%%%%%%%%%%%%%%%%%%%%%%%%%%%%%%%%%%%%%%
\section{Introduction}

Let $X$ be an oriented smooth $4$-manifold. A {\em Real structure} on $X$ is a smooth orientation preserving involution $\sigma : X \to X$. We say that an embedded oriented surface $\Sigma \to X$ is {\em Real} if $\sigma$ sends $\Sigma$ to itself orientation reversingly. One motivation for studying Real embedded surfaces comes from real algebraic geometry. If $S$ is a non-singular real algebraic surface then its complexification $S_{\mathbb{C}}$ is a complex surface with a Real structure $\sigma$ given by complex conjugation. A non-singular real algebraic curve in $S$ can be complexified giving a Real embedded surface in $S_{\mathbb{C}}$.

This paper is concerned with the study of Real embedded surfaces, particularly the {\em Real minimal genus problem}: given a class $\alpha \in H^2(X ; \mathbb{Z})$, what is the minimum genus of a Real embedded surface in $X$ representing $\alpha$? Before tackling this problem an even more fundamental problem must be addressed, namely when can a class $\alpha \in H^2(X ; \mathbb{Z})$ be represented by a Real embedded surface at all? Our first result gives a precise answer to this question.

\begin{theorem}
Let $X$ be a compact, oriented, smooth $4$-manifold and $\sigma$ a Real structure on $X$. A class $\alpha \in H^2(X ; \mathbb{Z})$ can be represented by a Real embedded surface $\Sigma \subset X$ if and only if $\alpha$ is in the image of the forgetful map $H^2_{\mathbb{Z}_2}(X ; \mathbb{Z}_-) \to H^2(X ; \mathbb{Z})$. Here $\mathbb{Z}_-$ denotes the equivariant local system which is the constant local system $\mathbb{Z}$ but where $\sigma$ acts as $-1$.
\end{theorem}

Furthermore, we give a precise description of the image of the forgetful map $H^2_{\mathbb{Z}_2}(X ; \mathbb{Z}_-) \to H^2(X ; \mathbb{Z})$ when $b_1(X) = 0$ and $\sigma$ is not free.

\begin{theorem}
Let $X$ be a compact, oriented, smooth $4$-manifold and $\sigma$ a Real structure on $X$. Suppose that $b_1(X) = 0$ and that $\sigma$ is not free. Then the forgetful map $H^2_{\mathbb{Z}_2}(X ; \mathbb{Z}_-) \to H^2(X ; \mathbb{Z})$ is injective with image the set of $\alpha \in H^2(X ; \mathbb{Z})$ such that $\sigma^*(\alpha) = -\alpha$.
\end{theorem}

Hence if $b_1(X)=0$ and $\sigma$ is not free, then a class $\alpha \in H^2(X ; \mathbb{Z})$ can be represented by a Real embedded surface if and only if $\sigma^*(\alpha) = -\alpha$.

With the existence problem settled we turn to the minimal genus problem. First recall the usual adjunction inequality from Seiberg--Witten theory (\cite{km,law,os}) states that if $X$ is a compact, oriented, smooth $4$-manifold with $b_+(X) > 1$ and $\mathfrak{s}$ is a spin$^c$-structure for which the Seiberg--Witten invariant of $(X,\mathfrak{s})$ is non-zero, then for any smoothly embedded compact oriented surface $\Sigma \subset X$ of genus $g>0$ and non-negative self-intersection, we have
\[
2g - 2 \ge | \langle c(\mathfrak{s}) , [\Sigma] \rangle | + [\Sigma]^2.
\]
If $X$ has simple type then the same inequality still holds without assuming non-negative self-intersection. If $g=0$ and $[\Sigma]$ is non-torsion, then $[\Sigma]^2 < 0$.

We prove a similar adjunction inequality for Real embedded surfaces, but where the assumption that $X$ has a non-vanishing Seiberg--Witten invariant is replaced by the assumption that $(X,\sigma)$ has a non-vanishing {\em Real Seiberg--Witten invariant}. We will recall the basic properties of the Real Seiberg--Witten invariant in Section \ref{sec:rsw}. More details can be found in \cite{tw, bar3}. See also \cite{nak1,nak,kato,li,mi,kmt,bh,mpt} for related works on Real Seiberg--Witten theory. There are two versions of the Real Seiberg--Witten invariant that we will make use of. The first is a mod $2$ invariant $SW_R(X,\mathfrak{s}) \in \mathbb{Z}_2$ and the second, which is only defined under some additional assumptions is an integer invariant $SW_{R,\mathbb{Z}}(X,\mathfrak{s}) \in \mathbb{Z}$ (well-defined up to an overall sign factor). We will prove two versions of the adjunction inequality, one for non-negative self-intersection and one for arbitrary self-intersection. Let $b_+(X)^{-\sigma}$ denote the dimension of the $-1$-eigenspace of the action of $\sigma$ on $H^+(X)$.

\begin{theorem}
Let $X$ be a compact, oriented, smooth $4$-manifold and $\sigma$ a Real structure on $X$. Assume that $b_+(X)^{-\sigma} > 1$. Let $\mathfrak{s}$ be a Real spin$^c$-structure such that $SW_R(X , \mathfrak{s}) \neq 0$ (or that $SW_{R,\mathbb{Z}}(X , \mathfrak{s})$ is defined and non-zero). Let $\Sigma \subset X$ be an embedded Real surface of genus $g$.
\begin{itemize}
\item[(1)]{If $g>0$ and $[\Sigma]^2 \ge 0$, then the adjunction inequality holds:
\[
2g-2 \ge | \langle c(\mathfrak{s}) , [\Sigma] \rangle | + [\Sigma]^2.
\]
}
\item[(2)]{If $g=0$ then $[\Sigma]^2 \le 0$. Furthermore if $\sigma$ does not act freely and $[\Sigma]$ is non-torsion, then $[\Sigma]^2 < 0$.}
\end{itemize}

\end{theorem}

In the case that $\sigma$ is free, our adjunction inequality reduces to the adjunction inequality proven by Nakamura \cite{nak}.

In the case of arbitrary self-intersection, we have the following:

\begin{theorem}
Let $X$ be a compact, oriented, smooth $4$-manifold and $\sigma$ a Real structure on $X$. Assume that $b_+(X)^{-\sigma} > 1$. Let $\mathfrak{s}$ be a Real spin$^c$-structure such that $SW_{R,\mathbb{Z}}(X , \mathfrak{s})$ is defined and non-zero. Let $\Sigma \subset X$ be an embedded Real surface of genus $g$. Assume that $\sigma$ does not act freely on $\Sigma$. Then
\[
2g \ge | \langle c(\mathfrak{s}) , [\Sigma] \rangle | + [\Sigma]^2.
\]
\end{theorem}

Note that while we do not need to assume $X$ is simple type for this inequality to hold, we do need to assume that the integer Real Seiberg--Witten invariant $SW_{R,\mathbb{Z}}(X,\mathfrak{s})$ is defined. In the case that $b_1(X) = 0$, this is equivalent to assuming the Real Seiberg--Witten moduli space for $(X,\mathfrak{s})$ has expected dimension zero.

An interesting feature of the (integral) Real Seiberg--Witten invariants is that they can be non-zero on connected sums, whereas the ordinary Seiberg--Witten invariant vanishes on any connected sum $X_1 \# X_2$ where both summands have $b_+ > 0$. Consequently one can easily find $4$-manifolds where the ordinary adjunction inequality is not applicable but the Real adjunction inequality holds. Using this type of argument we can construct many examples of involutions for which the Real minimal genus is strictly larger than then ordinary minimal genus. The following result is a demonstration of this.

\begin{theorem}
Let $X$ be one of the $4$-manifolds
\begin{itemize}
\item[(i)]{$\# a \mathbb{CP}^2 \# b \overline{\mathbb{CP}^2}$ where $a \ge 4$, $b \ge a+17$,}
\item[(ii)]{$\# a(S^2 \times S^2) \# b K3$ where $a,b \ge 1$.}
\end{itemize}
Then $X$ admits an involution $\sigma$ for which the following holds. Let $u \in H^2(X ; \mathbb{Z})$ be a non-zero class. Assume $u^2 \ge 0$ and in case (i) assume $u$ is a multiple of a primitive ordinary class. Then $u$ can be represented by a smoothly embedded compact oriented surface of genus at most $u^2/2$. However any Real embedded surface representing $u$ has genus at least $u^2/2 + 1$.
\end{theorem}

%%%%%%%%%%%%%%%%%%%%%%%%%%%%%%%%%%%%%%%%%%%%%%%%%%%%%%%%%%%%%%%%%%%%%%
\subsection{Structure of the paper}

The structure is the paper is as follows. In Section \ref{sec:realsurface} we prove various fundamental results concerning Real embedded surfaces, in particular Theorem \ref{thm:realise} which completely describes which classes in $H^2$ can be realised by Real embedded surfaces. In Section \ref{sec:rsw} we recall the necessary background on the Real Seiberg--Witten invariants including the behaviour of the invariants under blowup and connected sum. In Section \ref{sec:adj1} we prove the adjunction inequality for Real embedded surfaces of non-negative self-intersection (Theorem \ref{thm:adjunctionr}) and in Section \ref{sec:adj2} we prove the adjunction inequality for Real embedded surfaces with arbitrary self-intersection (Theorem \ref{thm:adjunctionr2}). Lastly in Section \ref{sec:ex} we consider examples of $4$-manifolds with Real structure and with non-vanishing Real Seiberg--Witten invariants. In particular we prove Theorem \ref{thm:minrg} giving many examples where the minimal genus for Real embedded surfaces is strictly larger than the minimal genus for arbitrary embedded surfaces. We finish with some examples where the exact value of the minimal genus for Real embedded surfaces can be calculated, or where upper and lower bounds can be found.

%%%%%%%%%%%%%%%%%%%%%%%%%%%%%%%%%%%%%%%%%%%%%%%%%%%%%%%%%%%%%%%
\noindent{\bf Acknowledgments.} The author was financially supported by an Australian Research Council Future Fellowship, FT230100092.

%%%%%%%%%%%%%%%%%%%%%%%%%%%%%%%%%%%%%%%%%%%%%%%%%%%%%%%%
\section{Real surfaces}\label{sec:realsurface}

Unless stated otherwise, a Real embedded surface is always assumed to be connected. If $\Sigma$ is a Real embedded surface in a $4$-manifold $X$ then the Real structure on $X$ restricts to an orientation reversing involution $\sigma$ on $\Sigma$. Since $\sigma$ is orientation reversing, the fixed point set of $\sigma$ is necessarily of odd codimension and is thus given by the union of $n$ disjoint embedded loops. The complement $\Sigma \setminus \Sigma^\sigma$ has either one or two components. Set $a=0$ if the complement is disconnected and $a=1$ if it is connected. Let $g$ denote the genus of $\Sigma$. It is known that up to diffeomorphism, the orientation reversing involutions on a genus $g$ surface are classified by the pair $(n,a)$. Furthermore the pair $(n,a)$ is realised by an orientation reversing involution if and only if the following conditions are satisfied (\cite[\textsection 3]{gh}):
\begin{itemize}
\item[(1)]{$0 \le n \le g+1$.}
\item[(2)]{If $n=0$ then $a=1$. If $n = g+1$ then $a=0$.}
\item[(3)]{If $a=0$ then $n = g+1 \; ({\rm mod} \; 2)$.}
\end{itemize}

Suppose $\Sigma \subset X$ is a Real embedded surface in a Real $4$-manifold $(X , \sigma)$. We can modify the topology of $\Sigma$ without changing the homology class $[\Sigma]$ as follows. First, we can always attach a pair of trivially embedded handles such that they are exchanged by $\sigma$. This increases $g$ by two without changing $n$. If $a=1$, then the resulting surface also has $a=1$. If $a=0$, the resulting surface can either have $a=0$ or $a=1$ depending on the location of the attached handles. Second, if the fixed point set of $\sigma$ acting on $X$ is non-empty and contains an embedded surface $S$, then we can attach a single handle to $\Sigma$ such that the handle meets $S$ in a circle and such that $\sigma$ maps the handle to itself. This operation increases $g$ and $n$ by $1$ and leaves $a$ unchanged. To see that this operation is possible, consider the quotient map $\rho : X \to X/\sigma$. If $\sigma$ has isolated fixed points then $X/\sigma$ is an orbifold, but $\Sigma_0 = \rho(\Sigma)$ is disjoint from the orbifold points. Note also that $\Sigma_0$ is a surface with boundary. Now remove a disc from $\Sigma_0$ and attach a cylinder with one end the boundary of the removed disc and the other end a circle on $\rho(S)$. Let $\Sigma'_0$ denote the resulting surface. Then setting $\Sigma' = \rho^{-1}(\Sigma)$, it follows that $\Sigma'$ is obtained from $\Sigma$ by attaching a handle and this handle meets $S$ in a circle.

Using the two operations described above, the genus of a Real embedded surface representing a given class in $\alpha \in H^2(X ; \mathbb{Z})$ can be increased arbitrarily (or by an arbitrary even amount in the case that $\sigma$ has only isolated fixed points). Therefore it is natural to consider the minimal possible genus of a Real embedded surface in $X$ representing $\alpha$. We denote the minimal genus by $g^{min}_R(\alpha)$. If $\alpha$ can not be represented by a Real surface then we set $g^{min}_R(\alpha) = \infty$. Since $\sigma$ restricts to an orientation reversing involution on $\Sigma$ an obvious necessary condition for $\alpha$ to be representable by a Real embedded surface is that $\sigma^*(\alpha) = -\alpha$.

\begin{remark}
One can ask a slightly more general problem than just the minimal genus. Given a class $\alpha \in H^2(X ; \mathbb{Z})$ what are the possible topological types $(g,n,a)$ of a Real surface representing $\alpha$?
\end{remark}

Let $\Sigma \subset X$ be an embedded Real surface. Then then inclusion map $i : \Sigma \to X$ is $\mathbb{Z}_2$-equivariant and hence we obtain a push-forward map in equivariant cohomology $i_* : H^*(\Sigma ; \mathbb{Z}_-) \to H^{*+2}(X ; \mathbb{Z}_-)$. Here $\mathbb{Z}_-$ dentoes the equivariant local system given by the sign representation of $\mathbb{Z}_2 = \langle \sigma \rangle$ on $\mathbb{Z}$. In particular, $\Sigma$ has an {\em equivariant fundamental class} defined by $[\Sigma]_{\mathbb{Z}_2} = i_*(1) \in H^2_{\mathbb{Z}_2}(X ; \mathbb{Z}_-)$. The image of $[\Sigma]_{\mathbb{Z}_2}$ under the forgetful map $H^2_{\mathbb{Z}_2}(X ; \mathbb{Z}_-) \to H^2(X ; \mathbb{Z})$ is the usual fundamental class $[\Sigma] \in H^2(X ; \mathbb{Z})$ (where we use Poincar\'e duality to regard the fundamental class as a cohomology class). In particular, a class $\alpha \in H^2(X ; \mathbb{Z})$ is representable by a Real embedded surface only if it lies in the image of the forgetful map $H^2_{\mathbb{Z}_2}(X ; \mathbb{Z}_-) \to H^2(X ; \mathbb{Z})$. We will soon show that the converse of this result is also true, but first we explain how equivariant fundamental classes are related to Real line bundles.

A {\em Real line bundle} on $X$ is a complex line bundle $L \to X$ and a lift $\sigma_L : L \to L$ of $\sigma$ to an antilinear involution on $L$. We use a capital R to distinguish from real line bundles in the ordinary sense. It is known that Real line bundles are classified by $H^2_{\mathbb{Z}_2}(X ; \mathbb{Z}_-)$ (here we are making use of the isomorphism \cite{sti} between Borel equivariant cohomology and equivariant sheaf cohomology. The fact that Real line bundles are classified by degree $2$ equivariant sheaf cohomology with values in $\mathbb{Z}_-$ follows from the exponential sequence in exactly the same way as for ordinary line bundles). This suggests a connection to Real embedded surfaces. A section $s : X \to L$ of $L$ is said to be Real if $\sigma_L \circ s = s \circ \sigma$. If $s$ is Real and transverse to the zero section, then the zero locus of $s$ is a Real surface (possibly disconnected).

\begin{proposition}\label{prop:assL}
Let $\Sigma$ be a (possibly disconnected) Real embedded surface in $X$. Then there is a Real line bundle $L \to X$ and a Real section $s$ of $L$ which is transverse to the zero section and whose zero locus is $\Sigma$. Moreover, the equivariant fundamental class of $\Sigma$ equals the equivariant first Chern class of $L$.
\end{proposition}
\begin{proof}
The construction of $L$ is a smooth analogue of the construction of the line bundle associated to a divisor in a complex manifold. Let $U_1$ be a $\sigma$-invariant tubular neighbourhood of $\Sigma$ and $U_2 = X \setminus \Sigma$. We can identify $U_1$ with a disc bundle in the normal bundle $N$ on $\Sigma$. The derivative of $\sigma$ defines a Real structure on $N_\Sigma$. Let $L_1 \to U_1$ be the pullback of $N$ to $U_1$, equipped with the pullback Real structure. Since $U_1$ is a subset of $N$, we obtain a tautological section $s_1 : U_1 \to L_1$ which is clearly a Real section of $L_1$. Moreover $s_1$ is transverse to the zero section and the zero locus is precisely $\Sigma$. Now set $L_2 = \mathbb{C}$ the trivial line bundle with real struture given by complex conjugation. Set $U_{12} = U_1 \cap U_2$. Then $s_1$ is non-vanishing on $U_{12}$ and defines an isomorphism $L_2 |_{U_{12}} \buildrel s_1 \over \longrightarrow L_1|_{U_{12}}$ of Real line bundles on $U_{12}$. Hence we can use $s_1|_{U_{12}}$ to patch together $L_1$ and $L_2$ defining a Real line bundle $L$ on $X$. Moreover we have a Real section $s : X \to L$ of $L$ which on $U_1$ is given by $s_1$ and on $U_2$ is given by the constant section $1$. Clearly $s$ is transverse to the zero section and the zero locus is $\Sigma$.

Let $\alpha \in H^2_{\mathbb{Z}_2}(X ; \mathbb{Z}_-)$ denote the equivariant fundamental class of $\Sigma$ and let $c \in H^2_{\mathbb{Z}_2}(X ; \mathbb{Z}_-)$ denote the equivariant first Chern class of $L$. We will show that $\alpha = c$. Recall that $\alpha = i_*(1)$ where $i$ is the inclusion $i : \Sigma \to X$. Let $N \to \Sigma$ be the normal bundle of $\Sigma$. Let $DN$ denote the closed unit disc bundle of $N$ and $DN^*$ the complement in $DN$ of the zero section. Let $\tau_N \in H^2_{\mathbb{Z}_2}( DN , DN^* ; \mathbb{Z}_-)$ denote the Thom class of $N$. Identify $DN$ with a tubular neighbourhood of $\Sigma$ and let $\iota : DN \to X$ be the inclusion map. Let $X_0 = X \setminus \Sigma$. Then $\alpha = i_*(1)$ is given by the image of $\tau_N$ under the composition of excision isomorphism $H^2_{\mathbb{Z}_2}(DN , DN^* ; \mathbb{Z}_-) \to H^2_{\mathbb{Z}_2}(X , X_0 ; \mathbb{Z}_-)$ followed by the forgetful map $H^2_{\mathbb{Z}_2}(X , X_0 ; \mathbb{Z}_-) \to H^2_{\mathbb{Z}_2}(X ; \mathbb{Z}_-)$. Put differently, define $\widetilde{\alpha} \in H^2_{\mathbb{Z}_2}(X , X_0 ; \mathbb{Z}_-)$ to correspond to $\tau_N$ under excision, that is, $\iota^*(\widetilde{\alpha}) = \tau_N$. Then $\alpha$ is the image of $\widetilde{\alpha}$ under the forgetful map. Next, let $\tau_L \in H^2_{\mathbb{Z}_2}(DL , DL^* ; \mathbb{Z}_-)$ denote the Thom class of $L$. The zero section $s_0$ of $L$ defines a map of pairs $s_0 : (X , \emptyset) \to (DL,DL^*)$ and $c = s_0^*(\tau_L)$. Recall that we constructed a section $s : X \to L$ whose zero locus is $\Sigma$. Furthermore $s$ is valued in $DL$, so defines a map $s : (X  , \emptyset) \to (DL,DL^*)$. Since $s$ and $s_0$ are homotopic, we have $c = s^*(\tau_L)$. Note that $s|_{X_0}$ is non-vanishing so $s$ defines a map of pairs $s : (X,X_0) \to (DL , DL^*)$. Set $\widetilde{c} \in H^2_{\mathbb{Z}_2}(X , X_0 ; \mathbb{Z}_-)$ to be the pullback of $\tau_L$ under this map. Then $\widetilde{c}$ is a lift of $c$. To show that $\alpha = c$ it suffices to show that $\widetilde{\alpha} = \widetilde{c}$. In fact, since $\iota^* : H^2_{\mathbb{Z}_2}( X , X_0 ; \mathbb{Z}_-) \to H^2_{\mathbb{Z}_2}( DN , DN^* ; \mathbb{Z}_-)$ is an isomorphism, it suffices to show that $\iota^*(\widetilde{\alpha}) = \iota^*(\widetilde{c})$. Or since $\iota^*(\widetilde{\alpha}) = \tau_N$, we need to show that $\iota^*(\widetilde{c}) = \tau_N$.

From the construction of $L$ it follows that there is an isomorphism $N \cong L|_\Sigma$. Let $j : N \to L$ be the map corresponding under this isomorphism to the inclusion $L|_\Sigma \to L$. By restriction, we obtain a map $j : DN \to DL$. Since $N \cong L|_\Sigma$, it follows that $j^*(\tau_L) = \tau_N$. Next, observe that since $L|_\Sigma \cong N$, $DL|_{DN}$ can be identified with the fibre product $DN \times_{\Sigma} DN$. Then $j : DN \to DL$ is the map $j(u) = (0,u)$ and $s \circ \iota : DN \to DL$ is the map $s(u) = (u,u)$ (because $s$ was defined over $DN$ to be the tautological section). There is a homotopy $h_t : (DN , DN^*) \to (DL,DL^*)$ from $j$ to $s \circ \iota$ given by $h_t(u) = (tu,u)$. Hence $j^* = \iota^* \circ s^*$ as maps $H^2_{\mathbb{Z}_2}(DL , DL^* ; \mathbb{Z}_-) \to H^2_{\mathbb{Z}_2}(DN , DN^* ; \mathbb{Z}_-)$. Hence
\[
\tau_N = j^*(\tau_L) = \iota^* s^*(\tau_L) = \iota^*( \widetilde{c}),
\]
which is what we needed to show.
\end{proof}

\begin{theorem}\label{thm:realise}
A class $\alpha \in H^2(X ; \mathbb{Z})$ can be represented by a Real embedded surface if and only if $\alpha$ lies in the image of the forgetful map $H^2_{\mathbb{Z}_2}(X ; \mathbb{Z}_-) \to H^2(X ; \mathbb{Z})$.
\end{theorem}
\begin{proof}
We have already seen that this is a necessary condition. It remains to prove sufficiency. Assume that $\alpha$ admits a lift $\widetilde{\alpha} \in H^2_{\mathbb{Z}_2}(X ; \mathbb{Z}_-)$ to equivariant cohomology. Then $\widetilde{\alpha}$ defines a Real line bundle $L \to X$. We will show that $L$ admits a Real section which is transverse to the zero section. Assuming this to be the case, the zero locus $\Sigma$ of $s$ is a Real surface representing $\alpha$. If $\Sigma$ is not connected we can attach pairs of handles in a $\sigma$-invariant manner so as to obtain a connected Real surface representing the same class $\alpha$.

We prove that $L$ admits a Real section $s$ which is transverse to the zero section. Suppose the fixed point set of $\sigma$ consists of $k$ isolated points $x_1, \dots , x_k$ and $m$ embedded surfaces $S_1, \dots , S_m$. Let $\nu_1, \dots , \nu_k , \mu_1 , \dots , \mu_m$ be disjoint, closed $\sigma$-invariant tubular neighbourhoods of $x_1, \dots , x_k, S_1 , \dots , S_m$. Let $\nu = \nu_1 \cup \cdots \cup \nu_k \cup \mu_1 \cup \cdots \cup \mu_m$ be the union of these neighbourhoods and let $U$ be the complement of the interior of $\nu$. Note that $\sigma$ acts freely on $U$. We will construct $s$ in two steps. For the first step, we will show that there exists a Real section $s_0$ of $L|_{\nu}$ which is transverse to the zero section and moreover the restriction $s_0|_{\partial \nu}$ is also transverse to the zero section of $L|_{\partial \nu}$. Second, since $\sigma$ acts freely on $U$, $L|_U$ descends to a (rank $2$ real) vector bundle $V \to U/\sigma$ on the quotient. Since $s_0|_{\partial \nu}$ is Real and transverse to the zero section, it descends to a section $v$ of $V|_{\partial (U/\sigma)}$ which is transverse to the zero section. By standard results we can extend $v$ to a section of $V$ which is transverse to the zero section. Pulled back to $U$, $v$ defines a Real section of $L|_U$ which extends $s_0$ and which is transverse to the zero section, giving our sought-after section $s$.

To construct $s_0$, we must construct it over each component $\nu_1, \dots , \nu_k , \mu_1 , \dots , \mu_m$. Consider first the case of an isolated point $x_i$. Then $\nu_i$ can be identified with the closed unit disc $D$ in $\mathbb{R}^4$, $\sigma$ can be identified with the map $\sigma : D \to D$ given by $\sigma(x) = -x$ and $L|_D$ can be identified with the trivial Real line bundle $L|_D \cong D \times \mathbb{C}$, $\sigma_L(x,y) = (-x , \overline{y})$. We define $s_0|_{\nu_i}$ to be given by $s_0|_{\nu_i}(x) = (x,1)$. Clearly this is a Real section and it is non-vanishing so it transverse to the zero section for trivial reasons.

Lastly, consider the case of a embedded surface $S_j$. Then $\mu_j$ can be identified with the unit disc bundle in the normal bundle $N_j \to S_j$. The restriction of $\sigma$ by $\mu_j$ is the map which acts as $-1$ in the fibres of $N_j$. Restricted to $S_j$, $\sigma_L$ defines a Real structure on $L|_{S_j}$ in the usual sense, hence defines a real line bundle $A \to S_j$. It follows that $L|_{\mu_j}$ can be identified with the complexification of $A$. More precisely, the pullback of $A$ to $N_j$ (which we continue to denote by $A$) is an equivariant real line bundle in an obvious way and then $L|_{\mu_j}$ is isomorphic to $A \oplus A$, where $\sigma_L( a_1 , a_2 ) = ( \sigma(a_1) , -\sigma(a_2) )$. Let $b : S_j \to A$ be a section of $A$ which is transverse to the zero section. Hence the zero locus of $b$ consists of a disjoint union of embedded circles, $C_1, \dots , C_r$. Choose disjoint, closed tubular neighbourhoods $\lambda_1, \dots , \lambda_r \subset S_j$ of $C_1,\dots , C_r$. We can identify each $\lambda_l$ with a cylinder $\lambda_l \cong [-1,1] \times S^1$ such that $C_l = \{0 \} \times S^1$. Let $\psi_l : \Sigma \to \mathbb{R}$ be a smooth function which is zero outside of $\lambda_l$ and which is given in $\lambda_l \cong [-1,1] \times S^1$ by $\psi_l( u,v ) = \varphi(u)$, where $\varphi : [-1,1] \to \mathbb{R}$ is smooth, $\varphi(u) = 0$ for $|u| > 2/3$, $\varphi(u) = 1$ for $|u| < 1/3$. Since $b$ is non-vanishing on $\{ \epsilon \} \times S^1 \subset \lambda_l$ for $\epsilon \neq 0$, it follows that $A|_{\lambda_l}$ is trivial. Let $c_l$ be a non-vanishing section of $A|_{\lambda_l}$. Consider the restriction $N_j |_{\lambda_l}$ of the normal bundle $N_j$ to $\lambda_l$. Since $\lambda_l$ is homotopy equivalent to a circle, we have that $N_j|_{\lambda_l}$ is either trivial (so isomorphic to $\mathbb{R} \oplus \mathbb{R}$), or isomorphic to $\mathbb{R} \oplus M$, where $M \to S^1$ is the unique non-trivial line bundle. Either way, projection to the first summand defines a map $f_l : N_j|_{\lambda_l} \to \mathbb{R}$ such that $f_l \circ \sigma = -f_l$ and $f_l$ is transverse to $0$. Now we define $s_0|_{\mu_j}$ as follows. First recall that $L|_{\mu_j} \cong A \oplus A$ with $\sigma_L(a_1,a_2) = (\sigma(a_1) , -\sigma(a_2))$. Define $s_0|_{\mu_j} = ( b , \sum_l \psi_l f_l c_l)$. Then $s_0|_{\mu_j}$ is Real since $b, \psi_l, c_l$ are $\sigma$-invariant and $f_l \circ \sigma = -f_l$ and is easily seen to be transverse to the zero section and furthermore $s_0|_{\partial \mu_j}$ is also transverse to the zero section of $L|_{\partial \mu_j}$.
\end{proof}

\begin{remark}
The proof of Theorem \ref{thm:realise} actually shows that $\alpha$ can be realised by a Real surface $\Sigma$ that meets each surface component of $X^\sigma$ in at most one circle. This is because every real line bundle on a compact surface admits a section which is transverse to the zero section and whose zero set is either empty or a circle. Furthermore, we can assume $\Sigma$ is disjoint from any surface component of $X^\sigma$ which is a sphere, because every real line bundle on a sphere is trivial.
\end{remark}

\begin{proposition}
Let $\Sigma \to X$ be a Real embedded surface and let $\alpha \in H^2_{\mathbb{Z}_2}(X ; \mathbb{Z}_-)$ denote the equivariant fundamental class of $\Sigma$. Let $S \subset X$ be a surface component of the fixed point set of $\sigma$ on $X$. The mod $2$ homology class of $\Sigma \cap S$ in $H_1(S ; \mathbb{Z}_2)$ is Poincar\'e dual to $\alpha|_S$ under the isomorphism $H^2_{\mathbb{Z}_2}(S ; \mathbb{Z}_-) \cong H^1(S ; \mathbb{Z}_2)$.
\end{proposition}
\begin{proof}
By Proposition \ref{prop:assL}, there is a Real line bundle $L$ and a generic Real section $s$ of $L$ whose zero set is $\Sigma$ and $\alpha = c_1(L)$. The restriction $L|_S$ is the complexification of a real line bundle $M$, classified by $\alpha|_S \in H^2_{\mathbb{Z}_2}(S ; \mathbb{Z}_-) \cong H^1(S ; \mathbb{Z}_2)$. $s|_S$ is a generic section of $M$ and the zero locus of $s|_S$ is $\Sigma \cap S$. Hence $\Sigma \cap S$ is Poincar\'e dual to $w_1(M) = \alpha|_S$.
\end{proof}

In order to make use of Theorem \ref{thm:realise}, it is useful to have a description of the image of the forgetful map $H^2_{\mathbb{Z}_2}(X ; \mathbb{Z}_-) \to H^2(X ; \mathbb{Z})$. The following result does this.

\begin{theorem}\label{thm:eqh2}
Let $H^2(X ; \mathbb{Z})^{-\sigma} = \{ \alpha \in H^2(X ; \mathbb{Z}) \; | \; \sigma^*(\alpha) = -\alpha\}$. The image of the forgetful map $H^2_{\mathbb{Z}_2}(X ; \mathbb{Z}_-) \to H^2(X ; \mathbb{Z})$ is contained in $H^2(X ; \mathbb{Z})^{-\sigma}$ with finite cokernel. Every class of the form $\alpha = x - \sigma^*(x)$ for some $x \in H^2(X ; \mathbb{Z})$ is in the image. Furthermore if $b_1(X) = 0$ and $\sigma$ is not free then the image is equal to $H^2(X ; \mathbb{Z})^{-\sigma}$.
\end{theorem}
\begin{proof}
Consider the Borel spectral sequence $E_2^{p,q} = H^p( \mathbb{Z}_2 ; H^q(X ; \mathbb{Z}) \otimes \mathbb{Z}_-)$ abutting to $H^*_{\mathbb{Z}_2}(X ; \mathbb{Z}_-)$. Since $E_2^{2,0} = H^2( \mathbb{Z}_2 ; \mathbb{Z}_-) = 0$, we get a short exact sequence $0 \to E_3^{1,1} \to H^2_{\mathbb{Z}_2}(X ; \mathbb{Z}_-) \to E_4^{0,2} \to 0$. Furthermore $E_4^{0,2}$ is a subgroup of $E_2^{0,2} = H^2(X ; \mathbb{Z})^{-\sigma}$ and the composition $H^2_{\mathbb{Z}_2}(X ; \mathbb{Z}_-) \to E_4^{0,2} \to H^2(X ; \mathbb{Z})^{-\sigma}$ is the forgetful map. The image of the forgetful map is therefore $E_4^{0,2} \subseteq H^2(X ; \mathbb{Z})^{-\sigma}$. Since $E_2^{p,q}$ is finite for $p>0$, the same is true of $E_r^{p,q}$ for all $r$. Then the codomains of $d_2 : E_2^{0,2} \to E_2^{2,1}$ and $d_3 : E_3^{0,2} \to E_3^{3,0}$ are finite. Hence $E_4^{0,2} \to H^2(X ; \mathbb{Z})^{-\sigma}$ has finite cokernel.

Suppose $\sigma$ is not free. Let $x \in X$ be a fixed point of sigma. Then $x$ defines a section of the Borel fibration $X \times_{\mathbb{Z}_2} E\mathbb{Z}_2 \to B\mathbb{Z}_2$ and it follows that there can be no differentials mapping into the $q=0$ row of the spectral sequence. In particular, $d_3 : E_3^{0,2} \to E_3^{3,0}$ is zero. Suppose also that $b_1(X) = 0$. Then $E_2^{2,1} = 0$ and hence $d_2 : E_2^{0,2} \to E_2^{2,1}$ is zero. So under these assumptions we have that $E_4^{0,2} = E_2^{0,2} = H^2(X ; \mathbb{Z})^{-\sigma}$ so the image of the forgetful map $H^2_{\mathbb{Z}_2}(X ; \mathbb{Z}_-) \to H^2(X ; \mathbb{Z})$ is $H^2(X ; \mathbb{Z})^{-\sigma}$.

Lastly, suppose $\alpha = x - \sigma^*(x)$ for some $x \in H^2(X ; \mathbb{Z})$. Then $x = c_1(L)$ for some complex line bundle $L \to X$. It follows that $\alpha$ is the first Chern class of $M = L \otimes \sigma^*( \overline{L})$. But $M$ has an obvious Real structure and hence $\alpha$ lifts to a class in $H^2_{\mathbb{Z}_2}(X ; \mathbb{Z}_-)$.
\end{proof}

Let $M$ be a smooth manifold and $\sigma$ a smooth involution on $M$. A {\em Real structure} on a principal circle bundle $\pi : Y \to M$ is a lift $\sigma_Y$ of $\sigma$ to a smooth involution on $Y$ such that $\sigma_Y( yz) = \sigma_Y(y)\overline{z}$ for all $z \in S^1$ (where $S^1$ is taken to be the unit circle in $\mathbb{C}$). Clearly there is a correspondence between Real structures on $Y$ and Real structures on the associated line bundle $L = Y \times_{S^1} \mathbb{C}$. In particular, isomorphism classes of Real structures on $Y$ correspond to lifts of $c_1(Y) \in H^2(M ; \mathbb{Z})$ to $H^2_{\mathbb{Z}_2}(M ; \mathbb{Z}_-)$. Connections on $Y$ correspond to imaginary $1$-forms $i a \in i \Omega^1(Y)$ such that $a$ is $S^1$-invariant and $\iota_{\varphi} a = 1$, where $\varphi$ is the vector field on $Y$ generating the $S^1$-action, normalised so that the integral curves of $\varphi$ have period $2\pi$. Let $\sigma_Y$ be a Real structure on $Y$. A connection $ia$ on $Y$ is said to be {\em Real (with respect to $\sigma_Y$)} if $\sigma_Y^*(a) = -a$. This is equivalent to saying the horizontal distribution $Ker(a)$ is $\sigma_Y$-invariant, or that the associated covariant derivative on $L = Y \times_{S^1} \mathbb{C}$ commutes with the induced Real structure.

\begin{lemma}\label{lem:flatreal}
Let $\Sigma$ be a compact, connected, oriented surface and $\sigma$ an orientation reversing involution on $\Sigma$. Let $\pi : Y \to \Sigma$ be a principal circle bundle over $\Sigma$ of degree zero and let $\sigma_Y$ be a Real structure on $Y$. Then $Y$ admits a flat Real connection.
\end{lemma}
\begin{proof}
Flat Real line bundles on $\Sigma$ are classified by $H^1_{\mathbb{Z}_2}(\Sigma ; S^1_-)$, where $S^1_-$ is the equivariant constant sheaf with coefficient group $S^1$ and where $\sigma$ acts by combination of pullback and complex conjugation. Now consider the exponential sequence
\[
0 \to \mathbb{Z}_- \to \mathbb{R}_- \to S^1_- \to 0.
\]
This gives rise to an associated long exact sequence
\[
\cdots \to H^1_{\mathbb{Z}_2}(\Sigma ; S^1_-) \to H^2_{\mathbb{Z}_2}(\Sigma ; \mathbb{Z}_-) \to H^2_{\mathbb{Z}_2}(\Sigma ; \mathbb{R}_-) \to \cdots
\]
The forgetful map $H^2_{\mathbb{Z}_2}(\Sigma ; \mathbb{R}_-) \to H^2(\Sigma ; \mathbb{R})$ is easily seen to be an isomorphism (use the Borel spectral sequence). Therefore, the image of $H^1_{\mathbb{Z}_2}(\Sigma ; S^1_-) \to H^2_{\mathbb{Z}_2}(\Sigma ; \mathbb{Z}_-)$ is the kernel of the composition 
\[
H^2_{\mathbb{Z}_2}(\Sigma ; \mathbb{Z}_-) \to H^2(\Sigma ; \mathbb{Z}) \to H^2(\Sigma ; \mathbb{R}).
\]
Then since $H^2(\Sigma ; \mathbb{Z}) \to H^2(\Sigma ; \mathbb{R})$ is injective, this is also the kernel of $H^2_{\mathbb{Z}_2}(\Sigma ; \mathbb{Z}_-) \to H^2(\Sigma ; \mathbb{Z})$. Hence every class in the kernel of $H^2_{\mathbb{Z}_2}(\Sigma ; \mathbb{Z}_-) \to H^2(\Sigma ; \mathbb{Z})$ lies in the image of $H^1_{\mathbb{Z}_2}(\Sigma ; S^1_-) \to H^2_{\mathbb{Z}_2}(\Sigma ; \mathbb{Z}_-)$, which is to say, admits a flat Real connection. Since $Y$ is assumed to have degree zero, the class of $(Y , \sigma_Y)$ in $H^2_{\mathbb{Z}_2}(\Sigma ; \mathbb{Z}_-)$ lies in the kernel of $H^2_{\mathbb{Z}_2}(\Sigma ; \mathbb{Z}_-) \to H^2(\Sigma ; \mathbb{Z})$.
\end{proof}

Let $\Sigma$ be a compact, connected oriented surface of genus $g$ and let $\sigma$ be an orientation reversing involution. Choose a $\sigma$-invariant metric $g_\Sigma$ on $\Sigma$. The underlying conformal structure of $g_\Sigma$ makes $\Sigma$ into a Riemann surface and $\sigma$ is then an anti-holomorphic involution.

A divisor $D = \sum_i n_i x_i$ on $\Sigma$ is said to be {\em Real} if $\sigma D = D$, where $\sigma D = \sum_i n_i \sigma(x_i)$. Similar to the construction of Real line bundles given in the proof of Proposition \ref{prop:assL}, we can associate a Real structure on the line bundle $L = \mathcal{O}(D)$ and a Real meromorphic section $s$ whose divisor is $D$.

\begin{lemma}\label{lem:rdivgen}
Every Real line bundle on $\Sigma$ is isomorphic to a Real line bundle of the form $\mathcal{O}(D)$ for a Real divisor $D$.
\end{lemma}
\begin{proof}
Let $(E, \sigma_E)$ be a Real line bundle on $\Sigma$. By averaging we can construct a $\sigma_E$-invariant connection $\nabla$ on $E$. The $(0,1)$-part of $\nabla$ defines a holomorphic structure $E$ which is respected by $\sigma_E$ in the sense that $\sigma_E$ is anti-holomorphic. If $s$ is a holomorphic (or meromorphic) section of $E$, then so is $\sigma_E \circ s \circ \sigma$. Now choose a Real effective divisor $D$ such that $deg(D) \ge g - deg(E)$ and let $L = \mathcal{O}(D)$ be the associated holomorphic line bundle. Then by Riemann--Roch $h^0(E \otimes L) > 0$. The Real structures on $E$ and $L$ define a Real structure on $E \otimes L$ and this defines a real structure (in the usual sense) on the complex vector space $H^0(\Sigma , E \otimes L)$. It follows that $E \otimes L$ admits a Real holomorphic section $f$ which is not identically zero. Let $s$ be a Real holomorphic section of $L$ whose divisor is $D$. Then $f/s$ is a Real meromorphic section of $E$. It follows that $E$ is the Real line bundle associated to the divisor of $f/s$.
\end{proof}

Let $C_1, \dots , C_n$ denote the connected components of the fixed point set of $\sigma$. Thus each $C_i$ is an embedded circle in $\Sigma$. If $E$ is a Real line bundle on $\Sigma$, then $E|_{C_i}$ admits a real structure, so is the complexification of a real line bundle $E_i$ on $C_i$. Let $u_i(E) \in \mathbb{Z}_2$ denote the first Stiefel--Whitney class of $E_i$ (under the identification $H^1(C_i ; \mathbb{Z}_2) \cong \mathbb{Z}_2)$. 

\begin{proposition}\label{prop:deg2}
The homomorphism $\phi : H^2_{\mathbb{Z}_2}(\Sigma ; \mathbb{Z}_-) \to \mathbb{Z} \oplus \mathbb{Z}_2^n$ given by
\[
\phi(E) = ( deg(E) , u_1(E) , \dots , u_n(E))
\]
is injective with image equal to 
\[
A = \{ (d , u_1 , \dots , u_n) \in \mathbb{Z} \oplus \mathbb{Z}_2^n \; | \; d = u_1 + \cdots + u_n \; ({\rm mod} \; 2) \}.
\]
\end{proposition}
\begin{proof}
We first show that the image of $\phi$ is $A$. Let $x_i \in C_i$. Then $\phi( \mathcal{O}(x_i) ) = (1 , e_i )$, where $e_i$ denotes the $i$-th standard basis vector of $\mathbb{Z}_2^n$. Let $x \in \Sigma$ be a non-fixed point. Then $\phi( \mathcal{O}(x+\sigma(x)) ) = (2 , 0 , \dots , 0)$. Clearly the elements $(1 , e_i)$, $i = 1, \dots , n$ and $(2 , 0 , \dots , 0)$ belong to $A$ and in fact generate $A$. So the image of $\phi$ contains $A$. Since any Real divisor can be written as a linear combination of divisors of the form $x_i$, $x_i \in C_i$ or $x + \sigma(x)$, $\sigma(x) \neq x$, we see that $\phi( \mathcal{O}(D) ) \in A$ for any Real divisor $D$. Then by Lemma \ref{lem:rdivgen}, it follows that the image of $\phi$ is $A$.

Now we show $\phi$ is injective. If $n = 0$, then $\sigma$ acts freely. Let $M = \Sigma/\sigma$. Then $\Sigma \to M$ is the orientation double cover of $M$ and the equivariant local system $\mathbb{Z}_-$ descends via $\sigma$ to the orientation local system $\mathbb{Z}_{orn}$ on $M$. Hence
\[
H^2_{\mathbb{Z}_2}(\Sigma ; \mathbb{Z}_-) \cong H^2(M ; \mathbb{Z}_{orn}) \cong H_0(M ; \mathbb{Z}) \cong \mathbb{Z}.
\]
If $n=0$, then $A \cong \mathbb{Z}$. Hence $\phi$ can be indentified with a surjective homomorphism $\phi : \mathbb{Z} \to \mathbb{Z}$. This implies that $\phi$ is injective.

If $n > 0$, then the Borel spectral sequence yield a short exact sequence
\[
0 \to \mathbb{Z}_2^c \to H^2_{\mathbb{Z}_2}(\Sigma ; \mathbb{Z}_-) \to \mathbb{Z} \to 0.
\]
This sequence splits, so $H^2_{\mathbb{Z}_2}(\Sigma ; \mathbb{Z}_-) \cong \mathbb{Z} \oplus \mathbb{Z}_2^c$. Using localisation in equivariant cohomology one finds that $c = n-1$, so $H^2_{\mathbb{Z}_2}(\Sigma ; \mathbb{Z}_-) \cong \mathbb{Z} \oplus \mathbb{Z}_2^{n-1}$. Similarly, one sees that $A \cong \mathbb{Z} \oplus \mathbb{Z}_2^{n-1}$, so $\phi$ can be identified with a surjective homomorphism $\phi : \mathbb{Z} \oplus \mathbb{Z}_2^{n-1} \to \mathbb{Z} \oplus \mathbb{Z}_2^{n-1}$. Such a homomorphism is easily seen to be an isomorphism.
\end{proof}

Now if $\Sigma \to X$ is a Real embedded surface in a $4$-manifold $X$, then the normal bundle $N_\Sigma$ admits the structure of a Real line bundle. By Proposition \ref{prop:deg2}, the structure of $N_\Sigma$ is determined by the degree, which equals $[\Sigma]^2$ and the classes $u_i( N_\Sigma)$ for $i=1, \dots ,n$ where $C_1, \dots , C_n$ are the components of the fixed point set of $\sigma|_\Sigma$. Each circle $C_i$ must lie on a corresponding component $S \subset X$ of the fixed point set of $\sigma$ on $X$. Then $N_\Sigma|_{C_i}$ is easily seen to be the complexification of real line bundle $M_i$, the normal bundle of $C_i$ in $S$. Thus $M_i$ is trivial or non-trivial according to whether the loop $C_i$ preserves or reverses orientation in $S$. In particular, if $S$ is orientable then $M_i$ is automatically trivial.

\begin{proposition}\label{prop:n2}
Let $(X,\sigma)$ be a Real $4$-manifold and $\Sigma \to X$ a Real embedded surface. If each surface component of the fixed point set of $\sigma$ is orientable, then $[\Sigma]^2$ is even. Conversely if $[\Sigma]^2$ is odd, then $\Sigma$ intersects a non-orientable surface component of the fixed point set of $\sigma$.
\end{proposition}
\begin{proof}
Let $C_1, \dots , C_n$ be the components of the fixed point set of $\sigma|_{\Sigma}$. Then by Proposition \ref{prop:deg2}, we have $[\Sigma]^2 = u_1(N_\Sigma) + \cdots + u_n(N_\Sigma) \; ({\rm mod} \; 2)$. If each surface component of the fixed point set of $\sigma$ is orientable, then $u_i(N_\Sigma) = 0$ for all $i$ and hence $[\Sigma]^2$ is even. Conversely if $[\Sigma]^2$ is odd, then $u_i(N_\Sigma) = 1$ for some $i$, which means that $C_i$ defines a non-orientable loop in a surface component of the fixed point set of $\sigma$.
\end{proof}

%%%%%%%%%%%%%%%%%%%%%%%%%%%%%%%%%%%%%%%%%%%%%%%%%%%%%%%%
\section{Real Seiberg--Witten invariants}\label{sec:rsw}

In this section we review the notion of Real Seiberg--Witten invariants and recall some of their basic properties, following \cite{tw,bar3}.

Let $X$ be a compact, oriented, smooth $4$-manifold. Let $\sigma : X \to X$ be a smooth, orientation reversing involution. Let $g$ be a $\sigma$-invariant Riemannian metric on $X$. Let $c : Spin^c(4) \to Spin^c(4)$ be the involution which is trivial on $Spin(4)$ and is given by complex conjugation on $S^1 \subset Spin^c(4)$. Let $\mathfrak{s}$ be a spin$^c$-structure on $X$ with corresponding principal $Spin^c(4)$-bundle $P \to X$ which is a lift of the $SO(4)$-frame bundle of $X$ to $Spin^c(4)$. A {\em Real structure} on $\mathfrak{s}$ is a lift $\widehat{\sigma}$ of $\sigma$ to $P$ covering the natural lift of $\sigma$ to the frame frame bundle and satisfying $\widehat{\sigma}(ph) = \widehat{\sigma}(p) c(h)$ for all $p \in P$, $h \in Spin^c(4)$ and such $\widehat{\sigma}^2 = -1$. As explained in \cite{bar3}, this is equivalent to giving an involutive, anti-linear lift of $\sigma$ to the spinor bundles $S^{\pm}$ respecting the Hermitian structure and respecting Clifford multiplication. In particular, the spinor bundles $S^{\pm}$ of a Real spin$^c$-structures are Real vector bundles and the determinant line $L = det(S^{\pm})$ is a Real line bundle. It follows that $c(\mathfrak{s}) \in H^2(X ; \mathbb{Z})$ lifts to a class in $H^2_{\mathbb{Z}_2}(X ; \mathbb{Z}_-)$.

A {\em configuration} for the Seiberg--Witten equations (with respect to $(X,\mathfrak{s},g)$) is a pair $(A , \psi)$ where $A$ is a spin$^c$-connection which projects to the Levi--Civita connection on the frame bundle of $X$ and $\psi$ is a positive spinor. Suppose now that $g$ is $\sigma$-invariant and that $\widehat{\sigma}$ is a Real structure on $\mathfrak{s}$. Then $\widehat{\sigma}$ induces an involution on the space of configurations and we say that $(A , \psi)$ is a {\em Real configuration} if it is preserved by this involution. Let $\eta \in i \Omega^+(X)$ be an imaginary self-dual $2$-form on $X$ satisfying $\sigma^*(\eta) = -\eta$. The {\em Real Seiberg--Witten equations} with respect to $(X , \sigma , g , \eta , \mathfrak{s} , \widehat{\sigma})$ are the usual Seiberg--Witten equations
\begin{align*}
D_A \psi &= 0, \\
F^+_A + \eta &= q(\psi)
\end{align*}
restricted to Real configurations. Define the {\em Real gauge group} $\mathcal{G}_R = \{ f : X \to S^1 \; | \; \sigma^*(f) = \overline{f} \}$. Then $\mathcal{G}_R$ acts on Real configurations by gauge transformation and this action restricts to an action on the set of solutions of the Seiberg--Witten equations. Let $b_+(X)^{-\sigma}$ denote the dimension of the space of harmonic self-dual $2$-forms $\omega$ on $X$ satisfying $\sigma^*(\omega) = -\omega$. If $b_+(X)^{-\sigma} > 0$, then for generic $\eta$ the moduli space $\mathcal{M}$ of $\mathcal{G}_R$-equivalence classes of Real solutions to the Seiberg--Witten moduli spaces is a compact smooth manifold. The Real Seiberg--Witten invariants of $(X,\sigma)$ are then defined by evaluating cohomology classes of $B\mathcal{G}_R$ over $\mathcal{M}$. If $b_+(X)^{-\sigma} > 1$, then the invariants do not depend on the choice of $(g,\eta)$, whereas if $b_+(X)^{-\sigma} = 1$, then the invariants depend on the choice of a chamber similar to the usual Seiberg--Witten invariants. For simplicity, we will restrict attention to the pure Real Seiberg--Witten invariant. We will also assume that $b_+(X)^{-\sigma} > 1$ so that chambers do not need to be considered. The (pure) Real Seiberg--Witten invariant comes in two forms: a mod $2$ invariant
\[
SW_R(X , \mathfrak{s}) \in \mathbb{Z}_2
\]
which is valued in $\mathbb{Z}_2$ and is always defined, and an integer invariant
\[
SW_{R,\mathbb{Z}}(X , \mathfrak{s}) \in \mathbb{Z}
\]
which reduces to $SW_R(X,\mathfrak{s})$ mod $2$, but is only defined under certain conditions. Sufficient conditions for the integer invariant to be defined are: 
\begin{itemize}
\item[(1)]{$(c(\mathfrak{s})^2 - \sigma(X))/8 = b_+(X)^{-\sigma}$,}
\item[(2)]{$b_+(X)^{-\sigma}$ is even,}
\item[(3)]{$b_1(X)^{-\sigma} = 0$,}
\end{itemize}
where $b_1(X)^{-\sigma}$ is the dimension of the space of harmonic $1$-forms $\alpha$ satisfying $\sigma^*(\alpha) = -\alpha$. Strictly speaking, the invariant $SW_{R,\mathbb{Z}}(X , \mathfrak{s})$ also depends on a choice of orientation of a certain line bundle over the Real Jacobian of $X$. Without making this choice $SW_{R,\mathbb{Z}}(X,\mathfrak{s})$ is only defined up to an overall sign ambiguity. If one prefers to have a genuine integer invariant, then the absolute value $| SW_{R,\mathbb{Z}}(X , \mathfrak{s})|$ should be taken.

In the definition of the Real Seiberg--Witten invariant, we considered only configurations $(A , \psi)$ where $A$ projects to the Levi--Civita connection on the frame bundle. However it will be convenient to consider a more general situation, following \cite{moy,os}. Let $g$ be a $\sigma$-invariant metric on $X$ and suppose $\nabla$ is a $\sigma$-invariant, metric compatible connection on $TX$. Then a configuration with respect to $(X , \mathfrak{s} , g , \nabla)$ is a pair $(A , \psi)$, where $A$ is a spin$^c$-connection on $\mathfrak{s}$ which projects to $\nabla$ on the frame bundle and $\psi$ is a positive spinor. Suppose $\mathfrak{s}$ is given a Real structure $\widehat{\sigma}$. Since $\nabla$ is assumed to be $\sigma$-invariant it follows that $\widehat{\sigma}$ induces an involution on the space of configurations and we define a Real configuration to be one that is fixed by this involution. Now we may consider the moduli space of Real solutions to the Seiberg--Witten equations with respect to $(X , \sigma , g , \nabla , \eta , \mathfrak{s} , \widehat{\sigma})$. One again proves that if $b_+(X)^{-\sigma} > 1$, then for generic $\eta$ the moduli space of solutions is a compact manifold (smoothness is a straighforward application of Sard--Smale. Compactness follows from \cite[\textsection 4]{os}). Evaluating cohomology classes of $\mathcal{G}_R$ on this moduli space gives invariants which are independent of the choice of $\nabla$, and hence agree with the usual definition of the Real Seiberg--Witten invariants. This follows by a straightforward adaptation of \cite[Theorem 4.6]{os} to the Real setting.

We now proceed to summarise some key properties of the Real Seiberg--Witten invariant. First we need to make a few remarks concerning connected sums. If $(X,\sigma)$ is a smooth $4$-manifold with a Real structure, then a necessary condition for $(X,\sigma)$ to admit a Real spin$^c$-structure is that the fixed point set of $\sigma$ contains no isolated points (first of all recall that if $\sigma$ preserves a spin structure and is odd in the sense that its lift to the spin bundle squares to $-1$, then the codimension of the fixed point set is $2$ mod $4$ \cite[Proposition 8.46]{ab}. The same result carries over to the setting of odd spin$^c$ structures because locally any spin$^c$-structure reduces to a spin structure). Thus we will mostly be interested in Real structures whose fixed point set is a collection of embedded surfaces. If $(X_1 , \sigma_1), (X_2, \sigma_2)$ are smooth $4$-manifolds with Real structures and if $\sigma_1,\sigma_2$ have non-isolated fixed points, then we can form the equivariant connected sum $(X,\sigma) = (X_1 \# X_2 , \sigma_1 \# \sigma_2)$ by removing open balls around non-isolated fixed points $x_1 \in X_1$, $x_2 \in X_2$ and identifying their boundaries equivariantly. In general the resulting involution $\sigma$ on $X$ will depend on the choice of fixed points $x_1,x_2$. If $\mathfrak{s}_1, \mathfrak{s}_2$ are Real spin$^c$-structures on $X_1,X_2$, then the connected sum spin$^c$-structure $\mathfrak{s} = \mathfrak{s}_1 \# \mathfrak{s}_2$ inherits a Real structure, which is well-defined up to isomorphism \cite[\textsection 9]{bar3}.

Let $(X , \sigma)$ be a smooth $4$-manifold with Real structure. Consider the blow-up operation in the Real setting. In fact there are two possible operations which could be regarded as the Real version of blowing up. First, we could blow up $X$ along a non-isolated fixed point $x$ of $\sigma$. We define the Real blowup of $X$ along $x$ to be the equivariant connected sum $(X \# \overline{\mathbb{CP}^2} , \sigma \# \tau)$, where $\tau$ is complex conjugation and where the connected sum is carried out by removing open balls around $x \in X$ and around a fixed point of $\tau$ and identifying their boundaries (it doesn't matter which fixed point of $\tau$ is chosen since the fixed point set of $\tau$ is connected). Since $\overline{\mathbb{CP}^2}$ is simply-connected and $\tau$ acts as $-1$ on $H^2( \overline{\mathbb{CP}^2} ; \mathbb{Z})$, it is easily shown that every spin$^c$-structure on $\overline{\mathbb{CP}^2}$ admits a Real structure, which is unique up to isomorphism. In particular, if $\mathfrak{s}$ is a Real spin$^c$-structure on $X$, then we obtain a Real spin$^c$-structure $\mathfrak{s}'$ on the blowup, given by $\mathfrak{s}' = \mathfrak{s} \# \kappa$, where $\kappa$ is a Real spin$^c$-structure on $\overline{\mathbb{CP}^2}$ with $c(\kappa)^2 = -1$ (there are two choices for $\kappa$). 

The second type of blowup operation that we wish to consider is the blowup of $(X, \sigma)$ at a pair of points $x , \sigma(x) \in X$, where $x$ is not a fixed point of $\sigma$. We define the Real blowup of $X$ along $x,\sigma(x)$, to be the connected sum $X' = X \# \overline{\mathbb{CP}^2} \# \overline{\mathbb{CP}^2}$, which carries an obvious Real structure $\sigma'$ exchanging the two $\overline{\mathbb{CP}^2}$ summands. If $\mathfrak{s}$ is a Real spin$^c$-structure on $X$ and $\kappa$ is any spin$^c$-structure on $\overline{\mathbb{CP}^2}$, then $\mathfrak{s}' = \mathfrak{s} \# \kappa \# \overline{\kappa}$ admits a Real structure in a natural way, where $\overline{\kappa}$ denotes the charge conjugate of $\kappa$. To see this, we can view $X'$ as being given by $X' = X_0 \cup_Y Z$, where $X_0$ is $X$ with open balls around $x$ and $\sigma(x)$ removed, $Z$ is the disjoint union of two copies of $\overline{\mathbb{CP}^2}$ minus an open ball and $Y$ is the disjoint union of two copies of $S^3$. The decomposition $X' = X_0 \cup_Y Z$ is $\mathbb{Z}_2$-equviariant, where the involution on $X_0$ is the restriction of $\sigma$ to $X_0$ and the involutions $\sigma_Y, \sigma_Z$ on $Y$ and $Z$ swap the two components. The Real spin$^c$-structure $\mathfrak{s}$ restricts to a Real spin$^c$-structure $\mathfrak{s}_{X_0}$ on $X_0$ and the Real spin$^c$-structure $\kappa \cup \overline{\kappa}$ on $\overline{\mathbb{CP}^2} \cup \overline{\mathbb{CP}^2}$ restricts to a Real spin$^c$-structure $\mathfrak{s}_Z$ on $Z$. Since $Y$ admits a unique Real spin$^c$-structure, the Real spin$^c$-structures $\mathfrak{s}_{X_0}, \mathfrak{s}_Z$ can be patched together along $Y$. Any automorphism of the unique Real spin$^c$-structure on $Y$ is given by a smooth map $f : Y \to S^1$ satisfying $\sigma_Y^*(f) = f$. Clearly $f$ extends to a smooth map $\widetilde{f} : Z \to S^1$ satisfying $\sigma_Z^*(\widetilde{f}) = \widetilde{f}$. This implies that up to isomorphism, the Real structure on $\mathfrak{s}'$ does not depend on the choice of isomorphism $\mathfrak{s}_{X_0}|_Y \cong \mathfrak{s}_{Z}|_Y$.

We are now ready to state the blowup formula:

\begin{proposition}\label{prop:blowup}
Let $X$ be a compact, oriented, smooth $4$-manifold and $\sigma$ a Real structure. Suppose that $b_+(X)^{-\sigma} > 1$. Let $\mathfrak{s}$ be a Real spin$^c$-structure on $X$. Let $X'$ be either the Real blowup of $X$ along a fixed point $x$, or the Real blowup of $X$ along a pair of non-fixed points $x,\sigma(x)$. Let $\mathfrak{s}' = \mathfrak{s} \# \kappa$ in the fixed case or $\mathfrak{s} \# \kappa \# \overline{\kappa}$ in the non-fixed case, where $\kappa$ is a spin$^c$-structure on $\overline{\mathbb{CP}^2}$ with $c(\kappa)^2 = -1$. Then $SW_R(X' , \mathfrak{s}') = SW_R(X , \mathfrak{s})$. Further, if $(c(\mathfrak{s})^2 - \sigma(X))/8 = b_+(X)^{-\sigma}$, $b_+(X)^{-\sigma}$ is even and $b_1(X)^{-\sigma} = 0$, then $|SW_{R,\mathbb{Z}}(X',\mathfrak{s}')| = |SW_{R,\mathbb{Z}}(X,\mathfrak{s})|$.
\end{proposition}
\begin{proof}
In the fixed case, this follows from \cite[Theorem 9.1]{bar3} for the mod $2$ case and \cite[Theorem 9.3]{bar3} for the integral case. Now consider the non-fixed case. We make use of Real Bauer--Furuta invariants, as in \cite{bar3}. Let $f$ denote the Real Bauer--Furuta invariant of $(X,\mathfrak{s})$. Then the Real Bauer--Furuta invariant of $(X' , \mathfrak{s}')$ is easily seen to be $f' = f \wedge h$, where $h$ is the ordinary $S^1$-equivariant Bauer--Furuta invariant of $(\overline{\mathbb{CP}^2} , \kappa)$, but restricted to $\mathbb{Z}_2 \subset S^1$. Then we have $|SW_{R,\mathbb{Z}}(X , \mathfrak{s})| = (1/2)|deg(f)|$ and $|SW_{R,\mathbb{Z}}(X' , \mathfrak{s}')| = (1/2) |deg(f')|$, by \cite[Proposition 5.9 (2)]{bar3}. So it suffices to show that $|deg(f)| = |deg(f')|$. Since $|deg(f')| = |deg(f)| |deg(g)|$, we just need to show that $deg(g) = \pm 1$. This follows because $b_+( \overline{\mathbb{CP}^2} ) = 0$ and $c(\kappa)^2 = \sigma( \overline{\mathbb{CP}^2})$ (see \cite[\textsection 3]{bar}, where the degree of the Bauer--Furuta map of a family of $4$-manifolds is computed and restrict to the case that the family is a point).
\end{proof}

The next result gives a sufficient condition for the Real Seiberg--Witten invariant to be non-vanishing.

\begin{proposition}\label{prop:loc}
Let $X$ be a compact, oriented, smooth $4$-manifold and $\sigma$ a Real structure with $b_1(X)^{-\sigma} = 0$ and $b_+(X)^{-\sigma} > 1$. Let $\mathfrak{s}$ be a spin$^c$-structure on $X$ that admits at least one Real structure. If the Seiberg--Witten invariant $SW(X , \mathfrak{s})$ is odd, then there exists a Real structure on $\mathfrak{s}$ for which $SW_R(X , \mathfrak{s})$ is non-zero.
\end{proposition}
\begin{proof}
This follows from \cite[Theorem 1.8]{bar3}, in the special case $\Delta = m = 0$.
\end{proof}

Proposition \ref{prop:loc} allows us to deduce non-vanishing of the Real Seiberg--Witten invariant from non-vanishing of the ordinary Seiberg--Witten invariant. However for applications to Real surfaces, we would like to have examples of $4$-manifolds where the ordinary Seiberg--Witten invariants are zero, but the Real Seibeg--Witten invariants are non-zero. The connected sum formula stated below implies that examples of such $4$-manifolds are readily obtained by taking connected sums. The connected sum formula shows that the Real Seiberg--Witten invariants of connected sums are often non-zero, in contrast with the ordinary Seiberg--Witten invariants which vanish on a connected sum $X_1 \# X_2$ with $b_+(X_1), b_+(X_2) > 0$.

The following connected sum formula is a special case of \cite[Theorem 1.12]{bar3}:

\begin{proposition}\label{prop:csf}
Let $X_1,X_2$ be compact, oriented, smooth $4$-manifolds and let $\sigma_1, \sigma_2$ be Real structures on them. Let $\mathfrak{s}_1, \mathfrak{s}_2$ be Real spin$^c$-structures on $X_1,X_2$ and suppose that for $i=1,2$ we have $(c(\mathfrak{s}_i)^2 - \sigma(X_i))/8 = b_+(X_i)^{-\sigma}$, $b_+(X_i)^{-\sigma}$ is even and positive and $b_1(X_i)^{-\sigma} = 0$. Suppose that $\sigma_1,\sigma_2$ act non-freely so that the equivariant connected sum $X_1 \# X_2$ exists. Then
\[
|SW_{R,\mathbb{Z}}(X_1 \# X_2 , \mathfrak{s}_1 \# \mathfrak{s}_2) | = 2 |SW_{R,\mathbb{Z}}(X_1 , \mathfrak{s}_1)| |SW_{R,\mathbb{Z}}(X_2 , \mathfrak{s}_2) |.
\]
\end{proposition}

%%%%%%%%%%%%%%%%%%%%%%%%%%%%%%%%%%%%%%%%%%%%%%%%%%%%%%%%
\section{The adjunction inequality}\label{sec:adj1}

In this section we prove the adjunction inequality for Real embedded surfaces of non-negative self-intersection.

\begin{theorem}\label{thm:adjunctionr}
Let $X$ be a compact, oriented, smooth $4$-manifold and $\sigma$ a Real structure on $X$. Assume that $b_+(X)^{-\sigma} > 1$. Let $\mathfrak{s}$ be a Real spin$^c$-structure such that $SW_R(X , \mathfrak{s}) \neq 0$ (or that $SW_{R,\mathbb{Z}}(X , \mathfrak{s})$ is defined and non-zero). Let $\Sigma \subset X$ be an embedded Real surface of genus $g$.
\begin{itemize}
\item[(1)]{If $g>0$ and $[\Sigma]^2 \ge 0$, then the adjunction inequality holds:
\[
2g-2 \ge | \langle c(\mathfrak{s}) , [\Sigma] \rangle | + [\Sigma]^2.
\]
}
\item[(2)]{If $g=0$ then $[\Sigma]^2 \le 0$. Furthermore if $\sigma$ does not act freely and $[\Sigma]$ is non-torsion, then $[\Sigma]^2 < 0$.}
\end{itemize}

\end{theorem}
\begin{proof}
(1) Assume first that $[\Sigma]^2 = 0$ and $g \ge 1$. Let $N_\Sigma$ denote the normal bundle of $\Sigma$, which is trivial since $[\Sigma]^2 = 0$. Let $g_0$ be a $\sigma$-invariant metric on $X$. For sufficiently small $\epsilon > 0$, the exponential map $exp : N_\Sigma \to X$ restricts to an embedding $exp : D_\epsilon \to X$, where $D_\epsilon$ is the open disc bundle in $N_\Sigma$ of radius $\epsilon$. Hence we can identify $D_\epsilon$ with an open neighbourhood of $\Sigma$ in $X$. Furthermore, $\sigma$ sends $D_\epsilon$ to itelf, since $g_0$ is $\sigma$-invariant. The action of $\sigma$ on $D_\epsilon$ covers the restriction $\sigma_\Sigma = \sigma|_{\Sigma} : \Sigma \to \Sigma$ of $\sigma$ to $\Sigma$ and acts by orientation reversing isometries on the fibres (orientation reversing because $\sigma$ reverses orientation on $\Sigma$ and preserves the orientation of $X$, hence must reverse orientation on the normal bundle). Let $Y$ denote the unit circle bundle in $N_\Sigma$. Since $\sigma$ preserves the metric on $N_\Sigma$, it induces an involution $\sigma_Y$ on $Y$. Then $D_\epsilon \cong [0,\epsilon) \times Y/\! \!\sim$ where $\sim$ collapses $\{0\} \times Y$ to $\Sigma$. Under this identification the action of $\sigma$ on $D_\epsilon$ is given by $\sigma(t,y) = (t , \sigma_Y(y))$. Let $g_\Sigma$ be a $\sigma$-invariant metric on $\Sigma$. Let $g_Y$ be a metric on $Y$ of the form $g_Y = \eta^2 + \pi^*(g_\Sigma)$, where $\pi : Y \to \Sigma$ is the projection and $i\eta$ is the connection $1$-form of a flat Real connection on $Y$, which exists by Lemma \ref{lem:flatreal}. By construction, $g_Y$ is $\sigma$-invariant. Let $g$ be a metric on $X$ whose restriction to the neck $[\epsilon/3 , 2\epsilon/3] \times Y \subset D_\epsilon$ equals a product metric of the form $dt^2 + g_Y$, where $t$ is the coordinate on $[0 , \epsilon)$. By averaging, we can assume that $g$ is $\sigma$-invariant. 

Over the neck, the tangent bundle of $X$ can be identified with a direct sum $TX|_{neck} \cong \mathbb{R} \oplus TY \cong \mathbb{R} \oplus \mathbb{R} \oplus T\Sigma$, where we use the connection $i\eta$ to split $TY$ into $\mathbb{R} \oplus T\Sigma$. Let $\nabla$ be a $g$-compatible connection on $TX$ which over the neck is of the form $\nabla_0 \oplus \nabla_\Sigma$, where $\nabla_0$ is the trivial connection on $\mathbb{R}^2$ and $\nabla_\Sigma$ is the pullback of the Levi--Civita connection on $\Sigma$ for $g_\Sigma$. By averaging we can assume $\nabla$ is $\sigma$-invariant.

For each real number $L > 0$, we obtain a $\sigma$-invariant metric $g(L)$ on $X$ by ``stretching the neck", which means to replace the neck $[\epsilon/3 , 2\epsilon/3] \times Y$ with $I \times Y$ where $I$ is an interval of length $L$. Similary, since $\nabla$ is translation invariant along the neck it defines a $g(L)$-compatible connection $\nabla(L)$ on $TX$ for all $L$. Since $SW_R(X , \mathfrak{s}) \neq 0$, there exists a solution to the Real Seiberg--Witten equations for $(X , \mathfrak{s})$ with respect to the metric $g(L)$, connection $\nabla(L)$ and zero perturbation. Letting $L \to \infty$ a standard neck-stretching argument \cite{km,os} implies that in the limit we get a solution to the Real Seiberg--Witten equations for $(Y , \mathfrak{s}|_Y)$ with respect to the metric $g_Y$, connection $\nabla_Y$ and with zero perturbation, where $\nabla_Y$ is the connection on $TY \cong \mathbb{R} \oplus T\Sigma$ which is trivial on the $\mathbb{R}$ summand and is $\nabla_\Sigma$ on the $T\Sigma$ summand.

Since $TX|_\Sigma \cong N_\Sigma \oplus T\Sigma$ and $N_\Sigma, T\Sigma$ are complex line bundles, this defines a complex structure on $TX|_\Sigma$. Let $\mathfrak{s}_{can}$ denote the canonical spin$^c$-structure associated to this complex structure. Any other spin$^c$-structure on $TX|_\Sigma$ is of the form $E \otimes \mathfrak{s}_{can}$ for some complex line bundle $E \to \Sigma$. In particular, we must have $\mathfrak{s}|_\Sigma \cong E \otimes \mathfrak{s}_{can}$ for some $E$. The negative spinor bundle for $\mathfrak{s}_{can}$ is given by $S_{can}^- = N_\Sigma \oplus K_\Sigma^{-1}$. Therefore the negative spinor bundle for $\mathfrak{s}|_\Sigma$ is of the form $S^-|_\Sigma = E \otimes (N_\Sigma \oplus K_\Sigma^{-1})$. Taking determinants gives $L|_\Sigma \cong E^2 \otimes N_\Sigma \otimes K_\Sigma^{-1}$, where $L$ denote the determinant line for $\mathfrak{s}$. Taking degrees gives $deg(L|_\Sigma) = 2 deg(E) + deg(N_\Sigma) + 2-2g$. Set $e = deg(E)$. Since we have assumed that $[\Sigma]^2 = 0$, we have $deg(N_\Sigma) = 0$. Then since $deg(L|_\Sigma) = \langle c(\mathfrak{s}) , [\Sigma] \rangle$, we have
\[
\langle c(\mathfrak{s}) , [\Sigma] \rangle = 2e + 2-2g.
\]
Since the identification $TX \cong N_\Sigma \oplus T\Sigma$ holds not just on $\Sigma$, but also on a tubular neighbourhood of $\Sigma$, it follows that it also holds on $Y$. In particular, the spinor bundle for the spin$^c$-structure $\mathfrak{s}|_Y$ is of the form $\pi^*( E \otimes (N_\Sigma \oplus K_\Sigma^{-1}) )$.

As shown in \cite[\textsection 5]{moy}, a solution to the Seiberg--Witten equations on $Y$ with respect to the metric $g_Y$, connection $\nabla_Y$ and zero perturbation is necessarily circle invariant and takes the form of a tuple $(B , \alpha , \beta)$, where $B$ is a Hermitian connection on $E$, $\alpha$ is a section of $E$ and $\beta$ is a section of $E \otimes N_\Sigma \otimes K_\Sigma^{*}$ satisfying
\begin{align*}
2F_B - F_{K_\Sigma} &= i( |\alpha|^2 - |\beta|^2 ) vol_\Sigma \\
\overline{\partial}_B \alpha &= 0 \\
\overline{\partial}_B^* \beta &= 0 \\
\alpha \overline{\beta} &= 0.
\end{align*}
For reducible solutions $\alpha = \beta = 0$ and then $2F_B = F_{K_\Sigma}$ which implies that $e = deg(E) = deg(K_\Sigma)/2 = g-1$. For irreducible solutions $\alpha \beta = 0$ implies that either $\alpha$ or $\beta$ is zero. If $\beta = 0$ and $\alpha \neq 0$, then $2F_B - F_{K_\Sigma} = i |\alpha|^2 vol_\Sigma$ implies that $e = deg(E) < deg(K_\Sigma)/2 = g-1$. But also $\alpha$ is a holomorphic section of $(E , \overline{\partial}_B)$, hence $deg(E) \ge 0$, so $0 \le e < g-1$. Similarly if $\alpha = 0$ and $\beta \neq 0$, then $\overline{\beta}$ is a holomorphic section of $K_\Sigma \otimes N_\Sigma^* \otimes E^*$ and one finds $g-1 < e \le 2g-2$. So in all cases we have $0 \le e \le 2g-2$. Hence
\[
| \langle c(\mathfrak{s}) , [\Sigma] \rangle | = 2| e +1-g | \le 2g-2
\]
which is the adjunction inequality in the case that $[\Sigma]^2 = 0$ and $g>0$.

Next, suppose that $g > 0$ and $[\Sigma]^2 > 0$. The idea is to use blowups to reduce to the case of zero self-intersection. However we need to be careful to ensure that the resulting blown-up surface is a Real surface. In the case that $\sigma|_\Sigma$ acts non-freely, there exists a point $p \in \Sigma$ which is a fixed point of $\sigma$. We blow up the point $p$. What this means is that we perform an equivariant connected sum of $(X , \sigma)$ with $(\overline{\mathbb{CP}^2} , \tau)$, where $\tau$ is complex conjugation. Letting $[x,y,z]$ denote homogeneous coordinates on $\overline{\mathbb{CP}^2}$ so that $\tau[x,y,z] = [\overline{x},\overline{y},\overline{z}]$, we see that the projective line $E \subset \overline{\mathbb{CP}^2}$ given by $z=0$ is a Real surface with respect to $\tau$. Now we perform the equivariant connected sum of $X$ and $\overline{\mathbb{CP}^2}$ by connecting $p$ and a fixed point of $E$, say $q = [1,0,0]$. Then the form the connected sum of surfaces $\Sigma' = \Sigma \# E$ which is a Real surface in $X \# \overline{\mathbb{CP}^2}$ and satisfies $[\Sigma']^2 = [\Sigma]^2 - 1$. Similarly we replace the Real spin$^c$-structure $\mathfrak{s}$ with $\mathfrak{s}' = \mathfrak{s} \# \kappa$, where $\kappa$ is a Real spin$^c$-structure on $\overline{\mathbb{CP}^2}$ with $c(\kappa)^2 = -1$. The blowup formula (Proposition \ref{prop:blowup}) implies that the Real Seiberg--Witten invariant of $(X' , \mathfrak{s}')$ is non-zero. Replacing $\kappa$ by $-\kappa$ if necessary, we can assume that $\langle c(\kappa) , [E] \rangle = \pm 1$ has the same sign as $\langle c(\mathfrak{s}) , [\Sigma] \rangle$. It follows that $|\langle c(\mathfrak{s}') , [\Sigma'] \rangle | = |\langle c(\mathfrak{s}) , [\Sigma] \rangle | + 1$ and hence
\[
| \langle c(\mathfrak{s}') , [\Sigma'] \rangle | + [\Sigma']^2 = | \langle c(\mathfrak{s}) , [\Sigma] \rangle | + [\Sigma]^2.
\]
Note also that if $\sigma$ acts non-freely on $\sigma$, then the induced Real structure on $X \# \overline{\mathbb{CP}^2}$ also acts non-freely on $\Sigma'$. Hence we can iterate this procedure until we obtain a Real surface of self-intersection zero.

In the case that $g > 0$, $[\Sigma]^2 > 0$ and $\sigma$ acts freely on $\Sigma$, instead of blowing up a fixed point, we will blow up a pair of points $x , \sigma(x)$ on $\Sigma$. The result will be a Real surface $\Sigma' \subset X \# 2 \overline{\mathbb{CP}^2}$ of the same genus as $\Sigma$ and with $[\Sigma']^2 = [\Sigma]^2 - 2$. The blowup formula (Proposition \ref{prop:blowup}) again implies that the Real Seiberg--Witten invariant of $(X' , \mathfrak{s}')$ is non-zero. We iterate this process until we obtain a surface of self-intersection zero. Since the self-intersection decreases by $2$ at each step, for this process to work it will be necessary that $[\Sigma]^2$ is even. But this is immediate from Proposition \ref{prop:deg2}.

Now we consider the case that $g=0$ and $[\Sigma]^2 > 0$. As in the $g>0$ case, we can blow up to get a $4$-manifold $X'$ and a Real surface $\Sigma'$ with $[\Sigma']^2 = 0$. Moreover, we can easily arrange that $\langle c(\mathfrak{s}') , [\Sigma]' \rangle \neq 0$. Indeed, each time we blow up the value of $\langle c(\mathfrak{s}') , [\Sigma'] \rangle$ changes by $+1$ or $-1$ depending on which spin$^c$-structure we put on the $\overline{\mathbb{CP}^2}$ summand. Since we are free to choose either spin$^c$-structure, we can always arrange that $\langle c(\mathfrak{s}') , [\Sigma'] \rangle \neq 0$. By the same neck-stretching argument as in the $g>0$ case, we get a solution to the Seiberg--Witten equations on $Y$, the unit circle bundle in the normal bundle of $\Sigma'$. However since $g=0$ all solutions are reducible and so $e = deg(E) = (1/2)deg(K_{\Sigma'}) = -1$ and hence $|\langle c(\mathfrak{s}') , [\Sigma'] \rangle | = 2(e+1) = 0$, a contradiction.

Lastly, suppose that $g=0$, $\sigma$ acts non-freely on $X$, $[\Sigma]^2 = 0$ and $[\Sigma]$ is non-torsion. We use an argument adapted from \cite[Lemma 5.1]{fs0}. Let $N_\Sigma$ denote the normal bundle of $\Sigma$. Since $[\Sigma]^2 = 0$, $N_\Sigma$ is trivial as an ordinary vector bundle. Furthermore, since $g=0$, the forgetful map $H^2_{\mathbb{Z}_2}( \Sigma ; \mathbb{Z}_-) \to H^2(\Sigma ; \mathbb{Z})$ is easily seen to be injective, hence $N_\Sigma$ is also trivial as a Real vector bundle. Hence $\Sigma$ has a tubular neighbourhood of the form $U = D \times \Sigma$ where $D$ is the (open) unit disc in $\mathbb{C}$ and $\sigma(x,y) = ( \overline{x} , \sigma|_\Sigma(y) )$. For any integer $n > 0$, take $n$ points $x_1, \dots , x_n \in D$ such that $x_1, \dots , x_n, \overline{x_1} , \dots, \overline{x_n}$ are distinct. This gives $2n$ disjoint copies of $\Sigma$, namely $\Sigma_i = \{x_i \} \times \Sigma$, $\Sigma_{n+i} = \{ \overline{x_i} \} \times \Sigma$, $1 \le i \le n$. For each $i = 1, \dots ,n$, $\sigma$ sends $\Sigma_i$ to $\Sigma_{n+i}$ orientation reversingly. Now let $p \in X$ be a fixed point of $\sigma$. Clearly $p \notin \Sigma_j$ for any $j$. Consider the blowup of $X$ at $p$. More precisely, let $\tau$ be the involution on $\overline{\mathbb{CP}^2}$ given by complex conjugation. Let $[x,y,z]$ be homogeneous coordinates on $\overline{\mathbb{CP}^2}$ so that $\tau[x,y,z] = [\overline{x},\overline{y},\overline{z}]$. Let $E \subset \overline{\mathbb{CP}^2}$ be the Real line given by $z=0$. By the blowup of $X$ at $p$, we mean the equivariant connected sum of $(X,\sigma)$ with $(\overline{\mathbb{CP}^2} , \tau)$ where the connected sum takes place at $p \in X$ and $[0,0,1] \in \overline{\mathbb{CP}^2}$. Since $[0,0,1] \notin E$, we can regard $E$ as an embedded sphere in $X' = X \# \overline{\mathbb{CP}^2}$. Furthermore $E$ is Real with respect to the involution $\sigma' = \sigma \# \tau$ on $X'$. Similarly, since $p \notin \Sigma_j$ for all $j$, we can regard $\Sigma_1, \dots , \Sigma_{2n}$ as embedded spheres in $X'$ and we have that $\sigma'$ sends $\Sigma'_i$ to $\Sigma'_{n+i}$ orientation reversingly for $i = 1, \dots , n$.

Let $\mathfrak{s}' = \mathfrak{s} \# \kappa$, where $\kappa$ is a Real spin$^c$-structure on $\overline{\mathbb{CP}^2}$ with $c(\kappa)^2 = -1$. Then $\mathfrak{s}'$ is a Real spin$^c$-structure on $X'$ and $SW_R(X',\mathfrak{s}') \neq 0$. Let $W' = X'/\sigma'$ be the quotient of $X'$ by $\sigma'$ (note that the existence of a Real spin$^c$-structure $\mathfrak{s}'$ on $X'$ implies that each component of the fixed point set of $\sigma'$ is an embedded surface and hence $W'$ admits the structure of a smooth $4$-manifold). Then $\rho : X' \to W'$ is a branched double cover. For $1 \le i \le n$, let $S_i = \rho(\Sigma_i)$ be the image of $\Sigma_i$ and let $F = \rho(E)$. Observe that since $\sigma$ acts on $E \cong \mathbb{CP}^1$ as complex conjugation, $F = E/\sigma'$ is an embedded disc in $W'$. Now for the connected sum $F_n$ of $F$ with $S_1, \dots , S_n$ along $n$ disjoint paths joining $n$ distinct points of $F$ to points on $S_1, \dots , S_n$. We can furthermore assume that the $n$ paths are chosen to be disjoint from the branch locus of $\rho : X' \to W'$. We have that $F_n$ is an embedded surface with genus $0$ and a single boundary component. Let $E_n$ be the preimage of $F_n$ in $W'$. Clearly $F_n$ is given by the connected sum of $E$ with $\Sigma_1, \dots , \Sigma_{2n}$ along $2n$ paths joining $2n$ points of $E$ to the surfaces $\Sigma_1, \dots , \Sigma_{2n}$. In particular $E_n$ is a closed embedded surface of genus $0$. Furthermore by construction $E_n$ is Real with respect to $\sigma'$ and $[E_n] = [E] + 2n[\Sigma]$. It follows that $[E_n]^2 = -1$. Thus $W'$ can be blown down along $E_n$. The result is a $4$-manifold $Z_n$ such that $W' = Z_n \# \overline{\mathbb{CP}^2}$. Moreover, there is an induced involution $\sigma_{Z_n}$ on $Z_n$ such that $\sigma' = \sigma_{Z_n} \# \tau$ and a Real spin$^c$-structure $\mathfrak{s}_n$ on $Z_n$ such that $\mathfrak{s}' = \mathfrak{s}_n \# \kappa$, where $c(\kappa) = [E_n]$. The blowup formula for the Real Seiberg--Witten invariants gives that $SW_R(Z_n , \mathfrak{s}_n) = SW_R(W' , \mathfrak{s}') \neq 0$. Applying the blowup formula to $\mathfrak{s}'_n = \mathfrak{s}_n \# (-\kappa)$, we instead get that $SW_R(W' , \mathfrak{s}_n \# (-\kappa)) \neq 0$. However we have that
\[
\mathfrak{s}' = \mathfrak{s}_n \# \kappa = [E_n] \otimes \mathfrak{s}'_n
\]
and hence
\[
c(\mathfrak{s}') = 2[E_n] + c(\mathfrak{s}'_n) = 2[E] + 4n[\Sigma] + c(\mathfrak{s}'_n).
\]
So
\[
c(\mathfrak{s}'_n) = c' -2[E] - 4n[\Sigma]
\]
where $c' = c(\mathfrak{s}')$. Since $[\Sigma]$ is assumed to be non-torsion, the classes $c(\mathfrak{s}'_n)$ represent infinitely many different classes in $H^2(W' ; \mathbb{Z})$ as $n$ ranges over the integers. However this is impossible as the compactness properties of the Seiberg--Witten equations implies that $SW_R(W' , \mathfrak{t} )$ can be non-zero for only finite many real spin$^c$-structures $\mathfrak{t}$. This is a contradiction and so $X$ does not admit a Real genus $0$ embedded surface $\Sigma$ with $[\Sigma]^2 = 0$ and $[\Sigma]$ non-torsion.
\end{proof}

%%%%%%%%%%%%%%%%%%%%%%%%%%%%%%%%%%%%%%%%%%%%%%%%%%%%%%%%%%%
\section{The adjunction inequality for negative self-intersection}\label{sec:adj2}

In this section we prove an adjunction inequality for Real embedded surfaces where the self-intersection is allowed to be negative. This can be compared with the adjunction inequality for ordinary embedded surfaces of negative self-intersection proven by Ozsv\'ath--Szab\'o \cite{os}. However our proof is very different and makes crucial use of the connected sum formula for Real Seiberg--Witten invariants.

\begin{theorem}\label{thm:adjunctionr2}
Let $X$ be a compact, oriented, smooth $4$-manifold and $\sigma$ a Real structure on $X$. Assume that $b_+(X)^{-\sigma} > 1$. Let $\mathfrak{s}$ be a Real spin$^c$-structure such that $SW_{R,\mathbb{Z}}(X , \mathfrak{s})$ is defined and non-zero. Let $\Sigma \subset X$ be an embedded Real surface of genus $g$. Assume that $\sigma$ does not act freely on $\Sigma$. Then
\[
2g \ge | \langle c(\mathfrak{s}) , [\Sigma] \rangle | + [\Sigma]^2.
\]
\end{theorem}
\begin{proof}

Choose a $d \ge 4$ such that $d + [\Sigma]^2 > 0$. Let $M \subset \mathbb{CP}^3$ be the non-singular degree $d$ hypersurface given by $-x_0^d + x_1^d + x_2^d + x_3^d = 0$. Complex conjugation sends $M$ to itself and hence defines an anti-holomorphic involution $\sigma_M$ on $M$. Let $\mathfrak{s}_M$ denote the canonical spin$^c$-structure on $M$. Then $\mathfrak{s}_M$ admits the structure of a Real spin$^c$-structure with respect to $\sigma_M$. Moreover $SW_{R , \mathbb{Z}}(M , \mathfrak{s}_M)$ is defined and non-zero (recall that for the canonical spin$^c$-structure of a K\"ahler surface, there is a unique solution to the Seiberg--Witten equations which is cut out transversally. This unique solution is easily seen to be Real and cut out transversally. Alternatively, one can use \cite[Theorem 1.7]{bar3}). Let $L = \mathcal{O}(1)|_M$. Since $\mathcal{O}(1) \to \mathbb{CP}^3$ is a Real line bundle with respect to complex conjugation, $L$ is a Real line bundle on $M$. Let $x = c_1(L)$. Then $x^2 = d$ because $M$ has degree $d$. The adjunction formula gives $K_M = \mathcal{O}(d-4)|_M = L^{d-4}$. Let $\Sigma_M \subset M$ be the real surface defined by $x_3 = 0$. Note that $x_3|_M$ is a Real section of $L$ and hence $\Sigma_M$ is a Real surface representing $x$. Hence $\langle c(\mathfrak{s}_M) , [\Sigma_M] \rangle = (d-4)x^2 = d(d-4)$. The genus of $\Sigma_M$ is given by $g_M = (d-1)(d-2)/2$, since $\Sigma_M$ can be viewed as the curve in $\mathbb{CP}^2$ defined by $x_0^d = x_1^d + x_2^d$. Note also that this equation has non-zero real solutions, so there exists fixed points of $\sigma_M|_{\Sigma_M}$.

Now consider the equivariant connected sum $(X' , \sigma') = (X , \sigma) \# (M , \sigma_M)$, where we connect sum a fixed point of $\Sigma$ to a fixed point of $\Sigma_M$. Let $\mathfrak{s}' = \mathfrak{s} \# \mathfrak{s}_M$. The connected sum formula Proposition \ref{prop:csf} implies that $SW_{R,\mathbb{Z}}(X' , \mathfrak{s}')$ is non-zero. Orient $\Sigma$ such that $\langle c(\mathfrak{s}) , [\Sigma] \rangle \ge 0$. Connect summing $\Sigma$ and $\Sigma_M$, we obtain a Real surface $\Sigma' = \Sigma \# \Sigma_M$ of genus $g' = g+g_M$ in $X'$. Since $[\Sigma']^2 = [\Sigma]^2 + d > 0$, the adjunction inequality Theorem \ref{thm:adjunctionr} gives
\[
2g' - 2 \ge | \langle c(\mathfrak{s}') , [\Sigma'] \rangle | + [\Sigma']^2
\]
Since $g' = g + (d-1)(d-2)/2$, $|\langle c(\mathfrak{s}') , [\Sigma'] \rangle | = | \langle c(\mathfrak{s}) , [\Sigma] \rangle + d(d-4) | = |\langle c(\mathfrak{s}) , [\Sigma] \rangle | + d(d-4)$ and $[\Sigma']^2 = [\Sigma]^2 + d$, we get
\[
2g + (d-1)(d-2) - 2 \ge |\langle c(\mathfrak{s}) , [\Sigma] \rangle | + d(d-4) + [\Sigma]^2 + d
\]
which simplifies to $2g \ge | \langle c(\mathfrak{s}) , [\Sigma] \rangle | + [\Sigma]^2$.
\end{proof}

\begin{remark}
In the setting of Theorem \ref{thm:adjunctionr2}, if $\sigma$ acts freely on $\Sigma$, but not freely on $X$, then as explained in Section \ref{sec:realsurface}, we can attach a handle to $\Sigma$ which increases $g$ by $1$ and adds a circle to the fixed point set of $\sigma|_\Sigma$. So without the assumption that $\sigma$ acts non-freely on $\Sigma$, we still get a slightly weaker inequality
\[
2g+2 \ge | \langle c(\mathfrak{s}) , [\Sigma] \rangle | + [\Sigma]^2.
\]
\end{remark}

%%%%%%%%%%%%%%%%%%%%%%%%%%%%%%%%%%%%%%%%%%%%%%%%%%%%%%%%%
\section{Examples}\label{sec:ex}

In this section we consider examples of $4$-manifolds with Real structure to which we can apply the Real adjunction inequality. We seek Real $4$-manifolds $(X , \sigma)$ with a non-vanishing Real Seiberg--Witten invariant and preferrably with vanishing ordinary Seiberg--Witten invariant, so that the usual adjunction inequality from Seiberg--Witten theory does not apply. As discussed in Section \ref{sec:rsw}, many examples can be obtained by taking equivariant sums. It follows from Proposition \ref{prop:csf} that if $(X_1 , \sigma_1), (X_2 , \sigma_1)$ are Real $4$-manifolds with $(c(\mathfrak{s}_i)^2 - \sigma(X_i))/8 = b_+(X_i)^{-\sigma}$ even and positive and $b_1(X_i)^{-\sigma} = 0$, then if $X_1,X_2$ have a non-vanishing Real integral Seiberg--Witten invariant then so does $(X_1 \# X_2 , \sigma_1 \# \sigma_2)$.

Following \cite{bar3}, we introduce the notion of an {\em admissible pair} $(X, \sigma)$. Admissible pairs are the building blocks from which many pairs with non-vanishing Real Seiberg--Witten invariant can be constructed by equivariant connected sums.

\begin{definition}
Let $X$ be a compact, oriented, smooth $4$-manifold and $\sigma$ an orientation preserving smooth involution on $X$. We will say the pair $(X , \sigma)$ is {\em admissible} if $b_1(X)^{-\sigma} = 0$, $\sigma$ has a non-isolated fixed point and $(X,\sigma)$ satisfies one of the following conditions:
\begin{itemize}
\item[(1)]{$X$ admits a symplectic structure with $\sigma^*(\omega) = -\omega$, and $b_+(X) - b_1(X) = 3 \; ({\rm mod} \; 4)$.}
\item[(2)]{$X$ is spin, $b_+(X)^{-\sigma} = \sigma(X) = 0$ and $H_1(X ; \mathbb{Z})$ has no $2$-torsion.}
\item[(3)]{$X$ has a spin$^c$-structure $\mathfrak{s}$ with $\sigma^*(\mathfrak{s}) = -\mathfrak{s}$, $SW(X,\mathfrak{s})$ is odd, $b_+(X) - b_1(X) = 3 \; ({\rm mod} \; 4)$, and
\[
\frac{ c(\mathfrak{s})^2 - \sigma(X) }{8} = \frac{ b_+(X) - b_1(X) + 1}{2} = b_+(X)^{-\sigma}.
\]
}
\item[(4)]{$X = \overline{\mathbb{CP}}^2$ with an involution such that $b_-(X)^{-\sigma} = 0$.}
\item[(5)]{$X = S^4 \# N \# N$ and $\sigma$ is obtained from the involution $diag(1,1,1,-1,-1)$ on $S^4$ by attaching two copies of $N$ which are exchanged by the involution, $N$ is negative definite, $b_1(N) = 0$ and there is a spin$^c$-structure on $N$ with $c(\mathfrak{s})^2 = -b_2(N)$.}
\end{itemize}

\end{definition}

Then we have the following result:

\begin{theorem}\label{thm:sumnonzero}
Let $(X,\sigma) = (X_1,\sigma_1) \# \cdots \# (X_k , \sigma_k)$ be an equivariant connected sum where $(X_i , \sigma_i)$ is admissible for each $i$. Suppose also that $b_+(X)^{-\sigma} > 1$. Then there is a Real spin$^c$-structure $\mathfrak{s}$ on $X$ such that $SW_{R,\mathbb{Z}}(X , \mathfrak{s})$ is defined and non-zero. If $\Sigma \subset X$ is a Real embedded surface in $X$ of genus $g$, $[\Sigma]^2 \ge 0$ and $[\Sigma]$ is non-torsion, then 
\[
g \ge \frac{1}{2}[\Sigma]^2 + 1.
\]
\end{theorem}
\begin{proof}
By \cite[Proposition 11.2]{bar3}, each summand $(X_i , \sigma_i)$ has a Real spin$^c$-structure $\mathfrak{s}_i$ such that $deg_R(X_i , \mathfrak{s}_i)$ is defined and non-zero (for the definition of $deg_R(X_i , \mathfrak{s}_i)$, see \cite[\textsection 5]{bar3}. Then the connected sum formula \cite[Theorem 1.11]{bar3} implies that $deg_R(X , \mathfrak{s})$ is defined and non-zero, where $\mathfrak{s} = \mathfrak{s}_1 \# \cdots \# \mathfrak{s}_k$. Since $b_+(X)^{-\sigma} > 1$, we have that $SW_{R,\mathbb{Z}}(X , \mathfrak{s}) = (1/2)deg_R(X,\mathfrak{s})$ by \cite[Proposition 1.9]{bar3} and hence $SW_{R,\mathbb{Z}}(X , \mathfrak{s})$ is non-zero. Therefore, the adjunction inequality Theorem \ref{thm:adjunctionr} gives that $g \neq 0$ and $2g-2 \ge |\langle c(\mathfrak{s}) , [\Sigma] \rangle | + [\Sigma]^2$. In particular, since $|\langle c(\mathfrak{s}) , [\Sigma] \rangle | \ge 0$, we get $g \ge (1/2)[\Sigma]^2 + 1$.
\end{proof}

\begin{remark}
If $(X,\sigma) = (X_1,\sigma_1) \# \cdots \# (X_k , \sigma_k)$ is an equivariant connected sum as in Theorem \ref{thm:sumnonzero} and if $b_+(X_i) > 0$ for at least two $i$, then the ordinary Seiberg--Witten invariants of $X$ vanish for all spin$^c$-structures.
\end{remark}

\begin{proposition}\label{prop:invol1}
Let $u \ge 4$, $v \ge u+17$, then $X = u \mathbb{CP}^2 \# v \overline{\mathbb{CP}^2}$ admits a Real structure $\sigma$ with $b_+(X)^{-\sigma} > 1$ and a Real spin$^c$-structure for which the Real Seiberg--Witten invariant is non-zero.
\end{proposition}
\begin{proof}
We will take $X$ to be an equivariant connected sum $X = K3 \# a(S^2 \times S^2) \# b \overline{\mathbb{CP}^2}$, where $a = u-3, b = v - u - 16$. We give $K3$ any odd involution which acts non-freely (for instance the involution which realises $K3$ as the branched double cover of a non-singular sextic in $\mathbb{CP}^2$). For the $S^2 \times S^2$ summands, we use the involution which swaps the two $S^2$ factors. This gives an involution on $X_0 = K3 \# a(S^2 \times S^2)$. We then blow up $X_0$ $b$ times. All the summands are admissible, so the existence of a Real spin$^c$-structure on $X$ with non-zero Real Seiberg--Witten invariant follows from Theorem \ref{thm:sumnonzero}.
\end{proof}

In a similar fashion we can get involutions on connected sums of $K3$ and $S^2 \times S^2$ with non-trivial Real Seiberg--Witten invariant.

\begin{proposition}\label{prop:invol2}
Let $a \ge 0$, $b \ge 1$. Then $X = a (S^2 \times S^2) \# b K3$ admits a Real structure $\sigma$ with $b_+(X)^{-\sigma} > 1$ and a Real spin$^c$-structure $\mathfrak{s}$ for which the Real Seiberg--Witten invariant is non-zero. Moreover we can choose $\mathfrak{s}$ such that the underlying spin$^c$-structure of $\mathfrak{s}$ comes from the unique spin structure on $X$.
\end{proposition}
\begin{proof}
We will construct $\sigma$ as an equivariant connect sum of involutions on $K3$ and $S^2 \times S^2$. On each $K3$ summand we use any odd involution with non-empty fixed point set. On each $S^2 \times S^2$ summand we use the involution which swaps the two $S^2$ factors. On each summand of the connected sum, we have an odd involution which then gives an Real spin$^c$-structure $\mathfrak{s}$ on $X$ whose underlying spin$^c$-structure comes from the unique spin structure. For the $K3$ summands, the spin structure is also the canonical spin$^c$-structure and it follows that the integral Real Seiberg--Witten invariant is defined and non-zero. For the $S^2 \times S^2$ summands, we have that $| deg_R( S^2 \times S^2 , \mathfrak{s}_{S^2 \times S^2} ) | = 1$ by \cite[Proposition 6.1 (3)]{bar3}, where $\mathfrak{s}_{S^2 \times S^2}$ is a Real spin$^c$-structure whose underlying spin$^c$-structure comes from the unique spin structure on $S^2 \times S^2$. Then by a version of the connected sum formula (\cite[Theorem 9.3 (1)]{bar3}), the Real--Seiberg Witten invariant of $(X , \mathfrak{s})$ is non-zero.
\end{proof}

To give examples of Real $4$-manifolds where the minimal genus for Real embedded surfaces is larger than the minimal genus of arbitrary embedded surfaces we will consider minimal genus in highly reducible $4$-manifolds, those of the form $a \mathbb{CP}^2 \# b \overline{\mathbb{CP}}^2$ or $a K3 \# b (S^2 \times S^2)$. Let $L$ be an integral lattice. A class $u \in L$ is {\em primitive} if it can not be written as $nu'$ for some integer $n > 1$ and some $u' \in L$. Every non-zero $u \in L$ can be uniquely be written as $u = du'$ for some integer $d \ge 1$ and some primitive $u' \in L$. We call $d$ the {\em divisibility} of $u$. A class $u \in L$ is called {\em ordinary} if it is not characteristic.

\begin{proposition}\label{prop:minred}
Let $X$ be one of the $4$-manifolds
\begin{itemize}
\item[(i)]{$\# a \mathbb{CP}^2 \# b \overline{\mathbb{CP}^2}$ where $a,b \ge 2$,}
\item[(ii)]{$\# a(S^2 \times S^2) \# b K3$ where $a,b \ge 1$.}
\end{itemize}
Let $u \in H^2(X ; \mathbb{Z})$ be a non-zero class. Assume $u^2 \ge 0$ and in case (i) assume $u$ is a multiple of a primitive ordinary class. Then $u$ can be represented by a smoothly embedded compact oriented surface of genus at most $u^2/2$.
\end{proposition}
\begin{proof}
By assumption $X$ is of the form $X = (S^2 \times S^2) \# M$, where $M$ is a compact, simply-connected, smooth $4$-manifold with indefinite intersection form. Let $\Gamma$ be the group of isomorphisms of $H^2(X ; \mathbb{Z})$ preserving the intersection form. A theorem of Wall \cite{wall} imples that every element of $\Gamma$ can be realised by an orientation preserving diffeomorphism of $X$. It follows that the minimal genus of a surface representing $u$ depends only on the orbit of $u$ under $\Gamma$. By another theorem of Wall \cite{wall0}, the $\Gamma$-orbit of $u$ depends only on the divisibility $d$, the norm $u^2$ and the type of $u/d$, i.e. whether $u/d$ is characteristic or ordinary. In case (ii) any primitive class is ordinary. In case (i) we have assumed that $u/d$ is ordinary. Thus in either case, the $\Gamma$-orbit of $u$ depends only on $d$ and $u^2$. Let $k = (u/d)^2$, so $u^2 = d^2k$.

Let $L$ be the intersection form of $X$. Write $X = (S^2 \times S^2) \# M$ so that $L = H \oplus L_M$ where $H = H^2(S^2 \times S^2 ; \mathbb{Z})$ is the hyperbolic lattice and $L_M$ is the intersection form of $M$. Give $H$ the basis $e,f$ where $e,f$ are the Poincar\'e duals of $S^2 \times \{ pt \}$ and $\{pt \} \times S^2$. So $e^2 = f^2 = 0$, $\langle e , f \rangle = 1$. 

If $k = 0$, then $u$ is in the same $\Gamma$-orbit as $v = de$. By \cite{rub}, $v$ can be represented by an embedded sphere in $S^2 \times S^2$, hence also by a sphere in $X$. Hence $u$ can be also represented by an embedded sphere.

If $k = 2m \ge 0$ is even and non-negative, then $u$ is in the same $\Gamma$-orbit as $v = de + dmf$. By \cite{rub}, $v$ can be represented by an embedded surface in $S^2 \times S^2$ of genus $g = (d-1)(dm-1)$. Hence also $u$ can be represented by an embedded surface in $X$ of the same genus.

If $k = 2m+1 > 0$ is odd and positive, then $L$ must be odd, so we are in case (i). Then we can write $X = (S^2 \times S^2) \# \mathbb{CP}^2 \# M'$ for some $M'$. Let $z \in L$ denote a generator of $H^2(\mathbb{CP}^2 ; \mathbb{Z})$. If $m=0$ so $k=1$. Then $u$ is in the same $\Gamma$-orbit as $v = dz$. By the degree-genus formula, $v$ can be represented by an embedded surface of genus $g = (d-1)(d-2)/2$, hence the same is true of $u$. If $m > 0$, then $u$ is in the same $\Gamma$-orbit as $v = v_1 + v_2$, where $v_1 = de + dmf$ and $v_2 = dz$. We have already seen that $v_1$ can be represented by a surface of genus $g_1 = (d-1)(dm-1)$ in $S^2 \times S^2$ and $v_2$ by a surface of genus $g_2 = (d-1)(d-2)/2$ in $\mathbb{CP}^2$. Connect summing these, $v$ can be represented in $X$ by a surface of genus $g = g_1 + g_2 = (d-1)(dm-1) + (d-1)(d-2)/2$, hence so can $u$.

In all of these cases one easily checks that $g \le d^2k/2 = u^2/2$, so the result is proven.
\end{proof}

Putting Theorem \ref{thm:sumnonzero} and Propositions \ref{prop:invol1}, \ref{prop:invol2} and \ref{prop:minred} together, we immediately get the following result. The result shows that the minimal genus for embedded surfaces and the minimal genus for Real embedded surfaces can be different.

\begin{theorem}\label{thm:minrg}
Let $X$ be one of the $4$-manifolds
\begin{itemize}
\item[(i)]{$\# a \mathbb{CP}^2 \# b \overline{\mathbb{CP}^2}$ where $a \ge 4$, $b \ge a+17$,}
\item[(ii)]{$\# a(S^2 \times S^2) \# b K3$ where $a,b \ge 1$.}
\end{itemize}
Then $X$ admits an involution $\sigma$ for which the following holds. Let $u \in H^2(X ; \mathbb{Z})$ be a non-zero class. Assume $u^2 \ge 0$ and in case (i) assume $u$ is a multiple of a primitive ordinary class. Then $u$ can be represented by a smoothly embedded compact oriented surface of genus at most $u^2/2$. However any Real embedded surface representing $u$ has genus at least $u^2/2 + 1$.
\end{theorem}

The adjunction inequality gives a lower bound on the minimal genus of Real embedded surfaces representing a given class. We can also hope to find upper bounds for the minimal genus via constructive methods. One such approach is to use real algebraic geometry.

\begin{proposition}\label{prop:vample}
Let $X$ be a compact K\"ahler surface and $\sigma$ an antiholomorphic involution on $X$. Let $x \in H^2(X ; \mathbb{Z})$. If $x = c_1(L)$ where $L$ is a very ample holomorphic line bundle and if $\sigma$ lifts to an antiholomorphic involution on $L$, then there exists a Real embedded surface $\Sigma$ representing $x$ of genus $g$ given by $2g-2 = -\langle c(\mathfrak{s}) , x \rangle + x^2$, where $\mathfrak{s}$ is the canonical spin$^c$-structure. If $b_+(X) \ge 3$ then $\Sigma$ has minimal genus.
\end{proposition}
\begin{proof}
Since $L$ is very ample, it defines a projective embedding $\varphi : X \to \mathbb{CP}^n$ for some $n$. The Real structure on $L$ induces a Real structure on $H^0(X , \mathcal{O}(L))$ and hence a Real structure $c$ on $\mathbb{CP}^n = \mathbb{CP}( H^0(X , \mathcal{O}(L))^*)$. In suitably chosen coordinates $c$ is given by complex conjugation. Moreover the map $\varphi$ is equivariant in the sense that $c \circ \varphi = \varphi \circ \sigma$. Bertinin's theorem implies that the intersection of $X$ with a generic hyperplane in $\mathbb{CP}^n$ is a non-singular connected subvariety of $X$. More precisely, the set of hyperplanes with this property is an open Zariski dense subset $U$ of the dual projective space $\mathbb{CP}( H^0(X , \mathcal{O}(L)))$. Therefore the intersection if $U$ with $\mathbb{CP}( H^0(X , \mathcal{O}(L)) )^{c}$ is non-empty. Hence there exsits a Real hyperplane whose intersection with $X$ is a Real, connected, embedded surface $\Sigma \subset X$. By construction $\Sigma$ is the zero locus of a section of $L$, so $[\Sigma] = x$. The adjunction formula gives $K_\Sigma = (K_X + L)|_\Sigma$, hence $2g-2 = -\langle c(\mathfrak{s}) , [\Sigma] \rangle + [\Sigma]^2 = -\langle c(\mathfrak{s}) , x\rangle + x^2$. If $b_+(X) \ge 3$, then the ordinary adjunction inequality implies that any embedded surface representing $x$ has genus $g'$ satisfying $2g'-2 \ge |\langle c(\mathfrak{s}) , x \rangle | + x^2 \ge -\langle c(\mathfrak{s}) , x \rangle + x^2 = 2g-2$. Hence $\Sigma$ has minimal genus.
\end{proof}

\begin{remark}
Note that if $b_1(X) = 0$ and $\sigma$ does not act freely, then a necessary and sufficient condition for a holomorphic line bundle $L$ to admit a Real structure compatible with the holomorphic structure is that $\sigma(x) = -x$. This follows because $\sigma(x) = -x$ implies that $\sigma^*(\overline{L}) \cong L$ as ordinary complex line bundles. But since $b_1(X) = 0$, a complex line bundle admits at most one holomorphic structure up to isomorphism, so $\sigma^*(\overline{L}) \cong L$ as holomorphic line bundles. This means that $\sigma$ admits an anti-holomorphic lift $\widetilde{\sigma} : L \to L$. Then $\widetilde{\sigma}^2 : L \to L$ is an automorphism of $L$ as a holomorphic line bundle, hence is multiplication by a non-zero constant $t$. Restricting to a fixed point of $\sigma$, we see that $t$ must be a positive real number. Replacing $\widetilde{\sigma}$ by $t^{-1/2} \widetilde{\sigma}$ gives an anti-holomorphic involutive lift of $\sigma$.
\end{remark}

\begin{example}
Let $X \subset \mathbb{CP}^{n}$ be a non-singular complete intersection of degree $(d_1 , d_2 , \dots , d_{n-2})$ such that the defining polynomials $f_1, \dots , f_{n-2}$ have real coefficients. To see that such complete intersections exist, we consider the space $P$ of all non-zero homogeneous polynomials $f_1 , \dots , f_{n-2}$ of degrees $d_1, \dots , d_{n-2}$ modulo rescalings $f_i \mapsto c_i f_i$, $c_i \in \mathbb{C}^*$. Then $P = \mathbb{CP}^{r_1} \times \cdots \times \mathbb{CP}^{r_{n-2}}$ is a product of projective spaces where $r_i = dim( H^0(\mathbb{CP}^n , \mathcal{O}(d_i) ) ) - 1$. Repeated application of Bertini's theorem implies that the set $\Delta \subset P$ of polynomials $(f_1 , \dots , f_{n-2})$ whose zero locus is singular forms a proper Zariski closed subset. Hence the complement $P \setminus \Delta$ intersects the locus of polynomials $(f_1 , \dots , f_{n-2})$ with real coefficients. If $f_1, \dots, f_{n-2}$ have real coefficients then complex conjugation on $\mathbb{CP}^n$ sends $X$ to itself and hence defines an anti-holomorphic involution $\sigma$ on $X$. For any $d > 0$, let $L = \mathcal{O}(d)|_X$. Then $L$ is very ample and $\sigma$ lifts to an anti-holomorphic involution on $L$. Then Proposition \ref{prop:vample} can be applied to $L$. The result is that if $b_+(X) \ge 3$, then minimal genus of a Real embedded surface representing $x = c_1(L)$ is given by
\[
2g-2 = dd_1 d_2 \cdots d_{n-2}( d + d_1 + d_2 + \cdots + d_{n-2} - n).
\]
\end{example}

\begin{example}
Let $(X , \sigma) = (X_1 , \sigma_1) \# \cdots \# (X_n , \sigma_n)$ be an equivariant connected sum where each $(X_i , \sigma_i)$ is as in Proposition \ref{prop:vample}. Assume further that $b_1(X_i) = 0$ for all $i$ and that $\sigma_i$ acts non-freely, so that the equivariant connected sum can be constructed. Let $L_i \to X_i$ be a very ample line bundle on $X_i$ for which $\sigma_i$ lifts to an anti-holomorphic involution. Let $x_i = c_1(L_i)$. Let $\mathfrak{s} = \mathfrak{s}_1 \# \cdots \# \mathfrak{s}_n$ where $\mathfrak{s}_i$ is the canonical spin$^c$-structure on $X_i$. By Proposition \ref{prop:vample}, $x_i$ can be represented by a Real embedded surface $\Sigma_i \subset X_i$ of genus $g_i$ satisfying $2g_i - 2 = -\langle c(\mathfrak{s}_i) , x_i \rangle + x_i^2$. Note that we must have $\langle c(\mathfrak{s}_i) , x_i \rangle \le 0$ because of the adjunction inequality $2g_i - 2 \ge | \langle c(\mathfrak{s}_i) , x_i) \rangle | + x_i^2$.

Now the idea is to connect $\Sigma_1, \dots , \Sigma_n$ together to form a Real surface representing $x = x_1 + \cdots + x_n$. However if we just perform an ordinary connected sum, the resulting surface will typically not be Real. If $\sigma_i$ acts non-freely on $\Sigma_i$ for each $i$, then we can do the connected sum equivariantly. Without this assumption we can still join $\Sigma_1$ to $\Sigma_2$ using a pair of handles that are exchanged by $\sigma$. Similarly we add a pair of handles joining $\Sigma_2$ to $\Sigma_3$ and so on. The result will be a Real embedded surface $\Sigma$ of genus $g = g_1 + g_2 + \cdots + g_n + (n-1)$. Thus
\begin{align*}
2g-2 &= \left( \sum_{i=1}^n 2g_i \right) + 2n-4 \\
&= \left( \sum_{i=1}^{n} (-\langle c(\mathfrak{s_i}) , x_i \rangle + x_i^2 + 2)\right) + 2n-4 \\
&= -\langle c(\mathfrak{s}) , x \rangle + x^2 + 4n-4 \\
&= | \langle c(\mathfrak{s}) , x \rangle | + x^2 + 4n-4.
\end{align*}
On the other hand, the connected sum formula for Real Seiberg--Witten invariants implies that $SW_{R,\mathbb{Z}}(X,\mathfrak{s}) \neq 0$ so we get a lower bound on $g$ from the adjunction inequality (here we use the assumption $b_1(X_i) = 0$ for all $i$ to ensure that the integral Real Seiberg--Witten invariants of $(X_i , \mathfrak{s}_i)$ are defined). Hence we have upper and lower bounds for the Real minimal genus $g^{min}_R(x)$
\[
|\langle c(\mathfrak{s}) , x \rangle| + x^2 \le 2g^{min}_R(x) - 2 \le |\langle c(\mathfrak{s}) , x \rangle | + x^2 + 4n-4.
\]
\end{example}

%%%%%%%%%%%%%%%%%%%%%%%%%%%%%%%%%%%%%%%%%%%%%%%%%%%%%%

\bibliographystyle{amsplain}

\begin{thebibliography}{99}
\bibitem{ab}M. F. Atiyah, R. Bott, A Lefschetz fixed point formula for elliptic complexes. II. Applications. {\em Ann. of Math.} (2) {\bf 88} (1968), 451-491. 
\bibitem{bar}D. Baraglia, Constraints on families of smooth 4-manifolds from Bauer-Furuta invariants. {\em Algebr. Geom. Topol.} {\bf 21} (2021) 317-349.
\bibitem{bar3}D. Baraglia, Exotic embedded surfaces and involutions from Real Seiberg--Witten theory. arXiv:2504.00281 (2025).
\bibitem{bh}D. Baraglia, P. Hekmati, New invariants of involutions from Seiberg--Witten Floer theory. arXiv:2403.00203 (2024).
\bibitem{fs0}R. Fintushel, R. Stern, Immersed spheres in $4$-manifolds and the immersed Thom conjecture. {\em Turkish J. Math.} {\bf 19} (1995), no. 2, 145-157.
\bibitem{gh}B. H. Gross, J. Harris, Real algebraic curves. {\em Ann. Sci. \'Ecole Norm. Sup.} (4) {\bf 14} (1981), no. 2, 157-182. 
\bibitem{kato}Y. Kato, Nonsmoothable actions of $\mathbb{Z}_2 \times \mathbb{Z}_2$ on spin four-manifolds. {\em Topology Appl.} {\bf 307} (2022), Paper No. 107868, 13 pp.
\bibitem{kmt}H. Konno, J. Miyazawa, M. Taniguchi, Involutions, links, and Floer cohomologies. {\em J. Topol.} {\bf 17} (2024), no. 2, Paper No. e12340, 47 pp.
\bibitem{km}P. B. Kronheimer, T. S. Mrowka, The genus of embedded surfaces in the projective plane. {\em Math. Res. Lett.} {\bf 1} (1994), no. 6, 797-808.
\bibitem{law}T. Lawson, The minimal genus problem. {\em Exposition. Math.} {\bf 15} (1997), no. 5, 385-431.
\bibitem{li}J. Li, Monopole Floer homology and real structures. arXiv:2211.10768 (2022).
\bibitem{mi}J. Miyazawa, A gauge theoretic invariant of embedded surfaces in 4-manifolds and exotic $P^2$-knots. arXiv:2312.02041 (2023).
\bibitem{mpt}J. Miyazawa, J. Park, M. Taniguchi, A satellite formula for real Seiberg-Witten Floer homotopy types. arXiv:2504.03270 (2025)
\bibitem{moy}T. Mrowka, P. Ozsv\'ath, B. Yu, Seiberg-Witten monopoles on Seifert fibered spaces. {\em Comm. Anal. Geom.} {\bf 5} (1997), no. 4, 685-791. 
\bibitem{nak1}N. Nakamura, $Pin^-(2)$-monopole equations and intersection forms with local coefficients of four-manifolds. {\em Math. Ann.} {\bf 357} (2013), no. 3, 915-939.
\bibitem{nak}N. Nakamura, $Pin^-(2)$-monopole invariants. {\em J. Differential Geom.} {\bf 101} (2015), no. 3, 507-549. 
\bibitem{os}P. Ozsv\'ath, Z. Szab\'o, The symplectic Thom conjecture. {\em Ann. of Math.} (2) {\bf 151} (2000), no. 1, 93-124. 
\bibitem{sti}A. Stieglitz, Equivariant sheaf cohomology. {\em Manuscripta Math.} {\bf 26} (1978/79), no. 1-2, 201-221.
\bibitem{rub}D. Ruberman, The minimal genus of an embedded surface of non-negative square in a rational surface. {\em Turkish J. Math.} {\bf 20} (1996), no. 1, 129-133.
\bibitem{tw}G. Tian, S. Wang, Orientability and real Seiberg-Witten invariants. {\em Internat. J. Math.} {\bf 20} (2009), no. 5, 573-604. 
\bibitem{wall0}C. T. C. Wall, On the orthogonal groups of unimodular quadratic forms. {\em Math. Ann.} {\bf 147} (1962), 328-338. 
\bibitem{wall}C. T. C. Wall, Diffeomorphisms of $4$-manifolds, {\em J. London Math. Soc.} {\bf 39} (1964) 131-140.
\end{thebibliography}

\end{document}